\documentclass[a4paper]{amsart}

\usepackage{amscd,amsthm,amsfonts,amssymb,amsmath}
\usepackage{fullpage}
\usepackage[all]{xy}

    \newcommand{\BA}{{\mathbb {A}}} 
    \newcommand{\BC}{{\mathbb {C}}} 
     \newcommand{\BF}{{\mathbb {F}}}
     
    \newcommand{\BI}{{\mathbb {I}}}

     \newcommand{\BP}{{\mathbb {P}}}
    \newcommand{\BQ}{{\mathbb {Q}}} \newcommand{\BR}{{\mathbb {R}}}
     \newcommand{\BT}{{\mathbb {T}}}

     \newcommand{\BZ}{{\mathbb {Z}}}

    \newcommand{\CO}{{\mathcal {O}}}

    \newcommand{\fa}{{\mathfrak{a}}} \newcommand{\fb}{{\mathfrak{b}}}

     \newcommand{\fp}{{\mathfrak{p}}}

     \newcommand{\fH}{{\mathfrak{H}}}

     \newcommand{\fN}{{\mathfrak{N}}}

    \newcommand{\Gal}{{\mathrm{Gal}}} 
    
    \newcommand{\Hom}{{\mathrm{Hom}}}

     \newcommand{\rank}{{\mathrm{rank}}}

    \renewcommand{\mod}{\ \mathrm{mod}\ }

    \newcommand{\Sel}{{\mathrm{Sel}}}

    \newcommand{\Spec}{{\mathrm{Spec}}}

    \font\cyr=wncyr10

    \newcommand{\Sha}{\hbox{\cyr X}}

    \theoremstyle{plain}
    \newtheorem{thm}{Theorem}[section] \newtheorem{cor}[thm]{Corollary}
    \newtheorem{lem}[thm]{Lemma}  \newtheorem{prop}[thm]{Proposition}
     \newtheorem{defn}[thm]{Definition}

\theoremstyle{remark} \newtheorem{remark}[thm]{Remark}
\theoremstyle{remark} 
\theoremstyle{remark} \newtheorem{example}{Example}

    \numberwithin{equation}{section}

\begin{document}

\title{Arithmetic of Eisenstein quotients}
\author{Yuan Ren}


\maketitle

\begin{abstract}
In this paper, we will study the arithmetic of the Eisenstein part of the modular Jacobians. In the first section, we introduce some general preliminaries of the arithmetic theory of modular curves that we will need later. In the second section, we give an example of modular abelian varieties due to Gross and study its properties in some details. In the third section, we define Eisenstein quotients of the modular Jacobians in general and give a criterion of the non-triviality of Heegner points on such Eisenstein quotients. The last two sections return to the concrete examples when the level of the modular Jacobian ia a prime or a square of a prime.
\end{abstract}

\tableofcontents

\section{Modular curves}
\subsection{Modular curves}
Let $\fH=\{z\in\BC:Im(z)>0\}$ be the upper half plane, $\fH=\fH$ and $Gl^{+}_2(\BR)=\{g\in Gl_2(\BR):det(g)>0\}$. There is an action of $Gl^{+}_2(\BR)$ on $\fH$ as
\[Gl^{+}_2(\BR)\times\fH\rightarrow\fH,(g,z)\mapsto gz\]
where $gz=\frac{az+b}{cz+d}$ for any $g=\left(
                                          \begin{array}{cc}
                                            a & b \\
                                            c & d \\
                                          \end{array}
                                        \right)
$.

A subgroup $\Gamma$ of $SL_2(\BZ)$ is called a congruence subgroup if $\Gamma\supseteq\Gamma(N)$ for some positive integer $N$. Here $\Gamma(N)$ is the subgroup of $SL_2(\BZ)$ which is congruent to $\left(
                                     \begin{array}{cc}
                                       1 & 0 \\
                                       0 & 1 \\
                                     \end{array}
                                   \right)
$ modulo $N$.

For example, for any positive integer $N$, the group
\[\Gamma_0(N):=\{\left(
                                          \begin{array}{cc}
                                            a & b \\
                                            c & d \\
                                          \end{array}
                                        \right)\in SL_2(\BZ):c=0\pmod N\}\]
is a congruence subgroup.

In the following, we always assume $\Gamma$ to be a congruence subgroup.

let $Y_{\Gamma}=\Gamma\setminus\fH$ be the quotient space, then it is shown in \cite{Sh} that there is a structure of Reimaa surface on $Y_{\Gamma}$. More over, let $\fH^*=\fH\bigcup\BP^1(\BQ)$ acted by $SL_2(\BZ)$ as the above formula, and define
\[X_{\Gamma}=\Gamma\setminus\fH^*\]

$X_{\Gamma}$ is a compact Riemann surface which is the compactification of $Y_{\Gamma}$. Let $S_{\Gamma}=\Gamma\setminus\BP^1(\BQ)$ and call it the set of cusps of $X_{\Gamma}$, then it is easy to see $S_{\Gamma}$ is a finite subset of $X_{\Gamma}$ and $X_{\Gamma}=Y_{\Gamma}\bigcup S_{\Gamma}$. We shall call $X_{\Gamma}$ the modular curve of level $\Gamma$.

It is known that a compact Riemann surface is algebraic over $\BC$ (GAGA). The important thing here is that these modular curves have algebraic models defined over number fields. Let's explains this for the curves $X_0(N):=\Gamma_0(N)\setminus\fH^*$.

\subsection{Moduli interpretation}
An elliptic curve $E$ over the complex number field $\BC$ is a Riemann surface of genus one. So as a complex manifold, it is of the form $\BC/L$ for some lattice $L$ and with the natural group structure, it is an abelian variety of dimension one over $\BC$. Two such abelian varieties $E_i=\BC/L_i$ ($i=1,2$) are isomorphic if and only if there is a number $\lambda\in\BC$ such that $L_2=\lambda L_1$, and the corresponding isomorphism is just the one induced by multiplication by $\lambda$. Hence we have a natural bijection between the upper half plane and the isomorphism classes of elliptic curves over $\BC$
\[SL_2(\BZ)\setminus\fH\rightarrow\{E/\BC\}/\thicksim\]
which sends $z\in\fH$ to the class represented by the elliptic curve $E_z=\BC/(\BZ\cdot z+\BZ)$.

More generally, we have the following
\begin{prop}\label{moduli}
For any positive integer $N$, here is a natural bijection between $Y_0(N)$ and the isomorphism classes of the pairs of elliptic curves over $\BC$ with a cyclic group of order $N$
\[Y_0(N)\rightarrow\{E/\BC\}/\thicksim\]
which sends $z\in\fH$ to the class represented by the elliptic curve $(E_z,<\frac{1}{N}>)$.
\end{prop}

More over, one can define elliptic curves with level structure algebraically, then the solution of the corresponding moduli problem will gives the desired model over canonically define number field (\cite{Mum}).

\begin{remark}
In Delinge and Rapapport's paper, they also gives a moduli intercalation of the compact modular curve in terms of generalized elliptic curves with level structure.
\end{remark}

\subsection{Hecke operators}
(modular forms and its relation to differential;Definition of Hecke operator;eigenform and Galois representation of it;modular abelian variety)
In this section, we write $X=X_0(N)/\BQ$, $J=J_0(N)/\BQ$ be its Jacobian and $i:X\rightarrow J$ the canonical morphism mapping $\infty$ to the zero. Recall that $Y=Y_0(N)=\Gamma_0(N)\setminus \fH$ is an open affine sub-scheme of $X$ which classifies the isomorphism class of pairs $[E,D]$, where $E$ is an elliptic curve and $D$ is a subgroup scheme isomorphic to $\BZ/N\BZ$.

For any two curves $C_1$, $C_2$ over some field $F$, a correspondence $T:C_1\rightsquigarrow C_2$ is by definition a triple $(C_3,\alpha,\beta)$, where $C_3$ is another curve and $\alpha$, $\beta$ are morphisms from $C_3$ to $C_1$ and $C_2$ respectively. From a correspondence $T$, one deduce two morphisms on Jacobians: the push forward $T_*:J(C_1)\rightarrow J(C_2)$ defined as $\beta_*\circ\alpha^*$ and the pull back $T^*:J(C_2)\rightarrow J(C_1)$ defined as $\alpha_*\circ\beta^*$.

For any prime $\ell$, let $X_0(N,\ell)$ to be the modular curve classifies the isomorphism classes $[E,D,C]$ with $E$ an elliptic curve, $D$ a subgroup scheme isomorphic to $\BZ/N\BZ$, $C$ a subgroup scheme isomorphic to $\BZ/\ell\BZ$ such that $D\bigcap C=0$. Define
\[\alpha_{\ell}:X_0(N,\ell)\rightarrow X, [E,D,C]\rightarrow[E,D]\]
and
\[\beta_{\ell}:X_0(N,\ell)\rightarrow X, [E,D,C]\rightarrow[E/C,D+C/C]\]

The Hecke correspondence $T_{\ell}$ is defined to be $(X_0(N,\ell),\alpha_{\ell},\beta_{\ell})$. As mentioned above, we will have two morphisms $T_{\ell,*}$ and $T^*_{\ell}$ on $J$.

\begin{prop}\label{two kinds of Hecke operator}
Notations as above, then we have
\[T_{\ell,*}=T^*_{\ell}=\left(
                          \begin{array}{cc}
                            \ell & 0 \\
                            0 & 1 \\
                          \end{array}
                        \right)
+\sum^{\ell-1}_{k=0}\left(
                                             \begin{array}{cc}
                                               1 & k \\
                                               0 & \ell \\
                                             \end{array}
                                           \right)
,\ell\nmid N\]
and when $\ell\mid N$, we have
\[T_{\ell,*}=\sum^{\ell-1}_{k=0}\left(
                                             \begin{array}{cc}
                                               1 & k \\
                                               0 & \ell \\
                                             \end{array}
                                           \right)\]
and
\[T^*_{\ell}=\sum^{\ell-1}_{k=0}\left(
                                  \begin{array}{cc}
                                    \ell & 0 \\
                                    Nk & 1 \\
                                  \end{array}
                                \right)
\]
\end{prop}
\begin{proof}
First, we assume $\ell\nmid N$. For any $(\BC/(\BZ\cdot z+\BZ),<\frac{1}{N}>)\in X$, we have
\[\alpha^{-1}(\BC/(\BZ\cdot z+\BZ),<\frac{1}{N}>)=(\BC/(\BZ\cdot z+\BZ),<\frac{1}{N}>+<\frac{1}{\ell}>)+\sum^{\ell-1}_{k=0}(\BC/(\BZ\cdot z+\BZ),<\frac{1}{N}>+<\frac{z+k}{\ell}>)\]
so we have
\[T_{\ell,*}(\BC/(\BZ\cdot z+\BZ),<\frac{1}{N}>)=(\BC/(\BZ\cdot\ell z+\BZ),<\frac{1}{N}>)+\sum^{\ell-1}_{k=0}(\BC/(\BZ\cdot\frac{z+k}{\ell}+\BZ),<\frac{1}{N}>)\]
Similarly, because we have
\[\beta^{-1}(\BC/(\BZ\cdot z+\BZ),<\frac{1}{N}>)=(\BC/(\BZ\cdot\ell z+\BZ),<\frac{1}{N}>+<z>)+\sum^{\ell-1}_{k=0}(\BC/(\BZ\cdot\frac{z+k}{\ell}+\BZ),<\frac{1}{N}>+<\frac{1}{\ell}>)\]
so we also have
\[T^*_{\ell}(\BC/(\BZ\cdot z+\BZ),<\frac{1}{N}>)=(\BC/(\BZ\cdot\ell z+\BZ),<\frac{1}{N}>)+\sum^{\ell-1}_{k=0}(\BC/(\BZ\cdot\frac{z+k}{\ell}+\BZ),<\frac{1}{N}>)\]

When $\ell\mid N$, the identity for $T_{\ell,*}$ is similar except we need to omit the term $(\BC/(\BZ\cdot z+\BZ),<\frac{1}{N}>+<\frac{1}{\ell}>)$ because $<\frac{1}{N}>\supseteq\frac{1}{\ell}$.

On the other hand, we have
\[\beta^{-1}(\BC/(\BZ\cdot z+\BZ),<\frac{1}{N}>)=\sum^{\ell-1}_{k=0}(\BC/(\BZ\cdot\ell z+\BZ),<\frac{1}{N}+kz>)=\sum^{\ell-1}_{k=0}(\BC/(\BZ\cdot\ell z+\BZ),<\frac{Nkz+1}{N}>)\]
so we have
\[\beta^{-1}(\BC/(\BZ\cdot z+\BZ),<\frac{1}{N}>)=\sum^{\ell-1}_{k=0}(\BC/(\BZ\cdot\ell z+\BZ),<\frac{Nkz+1}{N}>)\]
\[=\sum^{\ell-1}_{k=0}(\BC/(\BZ\cdot\ell z+\BZ\cdot Nkz+1),<\frac{Nkz+1}{N}>)\]
\[=\sum^{\ell-1}_{k=0}(\BC/(\BZ\cdot\frac{\ell z}{Nkz+1}+\BZ),<\frac{1}{N}>)\]
This proves our second claim.

\end{proof}

\begin{defn}
For any $N$, we define $\BT=\BZ[\{T_{\ell,*}\}_\ell]\subseteq End(J)$, and call it the (full) Hecke algebra of level $N$.
\end{defn}

For simplicity, we will write $T_{\ell}$ instead of $T_{\ell,*}$ in the following.

\section{On the 2-Selmer groups of the Gross curves}
In this section, we will study the $2$-Selmer group of some elliptic curves constructed by Gross in \cite{G3}. We will show in later sections the relation of these curves with some $2$-Eisenstein quotients of level a square of a prime.

\subsection{CM theory and descent}
Let $F$ be a field, an elliptic curve over $F$ is a smooth curve of genus one over $F$ with an $F$-rational point $O$. It is well known that such a curve admits a structure of abelian variety of dimension one such that $O$ is zero element.

Suppose $F$ is a number field, the Mordell-Weil theorem claims that the group of $F$-rational points $E(F)$ ia a finitely generated abelian group, so that $E(F)\simeq\BZ^{\oplus r}\bigoplus E(F)_{tor}$ with $E(F)_{tor}$ a finite group, for some non-negative integer $r$ called the rank of $E$ over $F$. In number theory, we are interested in determining the rank so that solve the Diophantine question. For this, there is the classical descent method.

Recall that from
  \[\xymatrix{
  0 \ar[r] & E(F)[n] \ar[r] & E(\bar{F}) \ar[r]^{n} & E(\bar{F}) \ar[r] &
  0}\]
we get the following diagram£º
  \[\xymatrix{
  0 \ar[r]& E()/{n E^{(d)}(p)}\ar[r]^\delta\ar[d]& H^1(G_F,E[n])\ar[r]\ar[d]&H^1(G_F,E[n])\ar[r]\ar[d]&0 \\
  0 \ar[r]&\prod E(F_v)/{n E^{(d)}(p)(F_v)}\ar[r]^{\delta_v}& \prod H^1(G_{F_v},E[n])\ar[r] &\prod H^1(G_{F_v},E)\ar[r]&0
  }\]

\begin{defn}
Define the n-Selmer group of $E$ over $F$ to be
\[\Sel_n(E/F)=ker(H^1(G_F,E[n])\rightarrow \prod H^1(G_{F_v},E))\]
and the Tate-Shafarevich group of $E$ over $F$ to be
\[\Sha(E/F)=ker(H^1(G_F,E)\rightarrow \prod H^1(G_{F_v},E))\]
\end{defn}
It follows that there is an exact sequence
$$
\xymatrix@C=0.5cm{
  0 \ar[r] & E(F)/{n E(F)} \ar[rr]^{\delta} && \Sel_{n}(E/F) \ar[rr]^{} && \Sha(E/F)[n] \ar[r] & 0 }
$$
The reason to introduce the Selmer groups is that $\Sel_{n}(E/F)$ is finite and relatively easy to compute, so that one can use them to obtain an upper bound of the rank. In the following, we will focus on the elliptic curves with complex multiplication and analyze the above exact sequence in some details for the so called $\BQ$-curves when $n=2$.

First we introduce the following notations:
\begin{itemize}
\item $K=$ an imaginary quadratic extension over $\BQ$;
\item $\CO=$ the integer ring of $K$;
\item $ELL(\CO)$=\{elliptic curve over $\BC$ with CM by $\CO\}$ up to
$C$-isomorphism;
\item $H=$ the Hilbert class field of $K$;
\item $ELL_H(\CO)$=\{elliptic curve over $H$ with CM by $\CO\}$ up to
$H$-isomorphism;
\item $ELL^\circ_H(\CO)=\{$elliptic curve over $H$ with CM by $\CO\}$ up to
$H$-isogeny;
\end{itemize}
Recall the following basic facts from CM theory, c.f. \cite{G}:

\begin{prop}
For any $E$ in $ELL_H(\CO)$, we have an associated continuous
homomorphism $\chi_E: \BA^\times_H\rightarrow{K^\times}$ such that

(i) $\chi|_{H^\times}=\mathbf{N}^H_K$, where $\mathbf{N}^H_K$ is the
norm map from $H$ to $K$;

(ii) $E$ has good reduction at $\beta\in \Spec\CO_H$ if and only if
$\chi_E$ is unramified at $\beta$. If $E$ has good reduction at
$\beta\in\Spec\CO_H$, then ${\chi_E}(\pi_\beta)$ is the unique
lifting of the $\mathbf{N}(\beta)$-th Frobenius of $
\widetilde{E}\pmod {\beta}$;

(iii) for any rational prime $\mathrm{\ell}$, we have
$\rho_\mathrm{\ell}={\chi_E}\cdot(\mathbf{N}^{H_\ell}_{K_\ell})^{-1}$,
where $\rho_\mathrm{\ell}:{G_H}\rightarrow {T_\ell}$ is the
$\ell$-adic Galois representation, ${H_\ell}=\prod_{w|\ell}{H_w}$
and $\mathbf{N}^{H_\ell}_{K_\ell}$ is the norm.
\end{prop}

We now review the $H$-isomorphic and $H$-isogenous classifications
of elliptic curves with CM by $\CO$.

\begin{thm}\label{thm:class} Let $J=\{j(E)\mid E\in{ELL(\CO)}\}$, and $\Sigma$ be the set of
continuous homomorphism $\chi:\BA^\times_H\rightarrow{K^\times}$ such
that $\chi|_{H^\times}=\mathbf{N}^H_K$, where $\mathbf{N}^H_K$ is
the norm map from $H$ to $K$, then

(i)There is a bijection
  \[{ELL_H(O)}\rightarrow{J\times\Sigma},\
  E/H\mapsto{(j(E),{\chi_E})};\]

(ii)There is a bijection
  \[{ELL^{\circ}_H(O)}\rightarrow{\Sigma},\
  E/H\mapsto{\chi_E}.\]
\end{thm}

\begin{lem}\label{Lem1}
For any $E\in\mathcal{ELL_H(O)}$ and
$\psi:{G_H}\rightarrow\CO^\times$ a continuous homomorphism,
let $E^\psi$ denote twist of $E$ by $\psi$ (note that
$\CO^\times=\mathrm{Aut}(E)$), then
$\chi_{E^\psi}=\psi\cdot\chi_E$.
\end{lem}

\begin{proof}
Recall ${E^\psi}$ is the $\psi$-twist of E means there is a
$\bar{H}$-isomorphism $\phi:{E}\rightarrow {E^\psi}$ such that for
any $g\in{G_H},\ \psi(g)=\phi^{-1}\circ\phi^g$. Fix such a $\phi$.

Let $w$ be a place where $\chi_{E^\psi}$,\ $\psi\cdot\chi_E$ and
$\psi$ are all unramified. Then from
$\psi^{\sigma(w)}=\phi^{-1}\circ\phi^{\sigma(w)}$, we have
$\psi^{\sigma(w)}=\phi^{-1}\circ\phi^{q_w}\pmod w$. So that as
morphism, we have
  \[\psi^{\sigma(w)}\circ[\chi_E(w)]=\phi^{-1}\circ[\chi_{E^\psi}(w)]\circ\phi\pmod w,\]
which implies that $\psi^{\sigma(w)}\cdot\chi_E(w)=\chi_{E^\psi}(w)$
by acting on the invariant differential.

As both $\chi_{E^\psi}=\psi\cdot\chi_E$ when restrict to
$\mathrm{H^\times}$, the approximation theorem implies that they are
the same on an open dense subset of $\BA^\times_H$ and the assertion
follows.
\end{proof}

\begin{lem}\label{lem2}
Let ${E_1},\ {E_2}\in\mathcal{ELL_H(O)},\phi:{E_1}\rightarrow{E_2}$
an $\bar{H}$-isogeny. Define
  \[\psi:{G_H}\rightarrow\CO^\times,g\mapsto\frac{\hat{\phi}\circ{\phi^g}}{\deg{\phi}}\in
    {H^1}({G_H},\CO^\times).\]
Then $\chi_{E_2}=\psi\cdot\chi_{E_1}$.
\end{lem}

\begin{proof}As in the proof of Lemma~\ref{Lem1}, for all but finitely many $w$,
because${\hat{\phi}\circ{\phi^{\sigma(w)}}}={\deg{\phi}}\cdot\psi(w)$,
${\hat{\phi}\circ{\phi^{q_w}}}={\deg{\phi}}\cdot\psi(w) \pmod w$,
which means
${\hat{\phi}\circ{[\chi_{E_2}(w)]}\circ{\phi^{q_w}}}={\deg{\phi}}\cdot\psi(w)\circ{\chi_{E_1}(w)}
\pmod w$.

Because $\hat{\phi}\circ{\phi}=\deg{\phi}$, acting on the invariant
differential gives that
$\deg{\phi}\cdot{\chi_{E_2}(w)}=\deg{\phi}\cdot\psi(w)\cdot{\chi_{E_1}(w)}$,
then we have $\chi_{E_2}=\psi\cdot\chi_{E_1}$
\end{proof}

\begin{proof}[Proof of Theorem~\ref{thm:class}] (c.f. \cite{G})
(i) $j(E_1)=j(E_2)$ implies that there is a $\bar{H}$-isomorphism
$\phi:{E_1}\rightarrow{E_2}$, so if define
$\psi:{G_H}\rightarrow\CO^\times,g\mapsto{\phi}^{-1}\circ{\phi}^g$
as in Lemma~\ref{Lem1}, then ${E_2}={E_1}^{\psi}$, and
$\chi_{E_2}=\psi\cdot\chi_{E_1}$. But then the assumption implies
that ${\psi}=1$ , i.e. $\phi$ is defined over H, so ${E_1}={E_2}$ in
$\mathcal{ELL_H(O)}$.

(ii) For any $E_1,E_2$, choose a $\phi\in \Hom(E_1,E_2)$ which is a
$\bar{H}$-isogeny, and define $\psi$ as in Lemma~\ref{lem2}, then
$\chi_{E_2}=\psi\cdot\chi_{E_1}$. Because $\chi_{E_1}=\chi_{E_2}$,
we have $\psi=1$, which means $\phi$ is defined over $H$.
\end{proof}

Recall the definition of $\BQ$-curves:
\begin{defn} $E\in{ELL_H(O)}$ is called a $\BQ$-curve, if for any
$\sigma\in \Gal(H/\BQ)$, we have $E^{\sigma}=E$ in
$\mathcal{ELL^{\circ}_H(O)}$.
\end{defn}

We will now describe the descent method used in \cite{G}.
\begin{lem}\label{Lem3} Let $E\in\mathcal{ELL_H(O)}$ be a $\BQ$-curve.
Then for any $\sigma\in G$,
  \[\Hom(E^\sigma,E)/{2\,\Hom(E^\sigma,E)}\cong{\mathcal {O}/2\mathcal {O}}.\]
\end{lem}

\begin{proof} Assume $E[2]$ is generated by
$P$ over ${\mathcal {O}/2 \mathcal {O}}$, so $E^\sigma[2]$ is
generated by $P^\sigma$. For any $\phi\in \Hom(E^\sigma,E)$, let
$[a_{\phi}]\in{O}/2 \mathcal {O}$ such that
$\phi(P^{\sigma})={a_{\phi}\cdot P}$, this gives a homomorphism
$\Hom(E^\sigma,E)/{2 \Hom(E^\sigma,E)}\rightarrow{\CO/2
\CO}$, which is obviously injective. On the other hand, the
density theorem implies that this homomorphism is surjective.
\end{proof}
Recall that from
  \[\xymatrix{
  0 \ar[r] & E(H)[2] \ar[r] & E(\bar{H}) \ar[r]^{2} & E(\bar{H}) \ar[r] &
  0}\]
we get the following diagram£º
  \[\xymatrix{
  0 \ar[r]& E(H)/{2 E^{(d)}(p)}\ar[r]^\delta\ar[d]& H^1(G_H,E[2])\ar[r]\ar[d]&H^1(G_H,E[2])\ar[r]\ar[d]&0 \\
  0 \ar[r]&\prod E(H_v)/{2 E^{(d)}(p)(H_v)}\ar[r]^{\delta_v}& \prod H^1(G_{H_v},E[2])\ar[r] &\prod H^1(G_{H_v},E[2])\ar[r]&0
  }\]

For any $\BQ$-curves $E$, we can give $E(H)/{2 E(H)}, \Sel_{2}(E/H)\
and\ \Sha(E/H)[2]$ a structure of $Gal(H/Q)$-module by using
Lemma~\ref{Lem3} as following

\begin{enumerate}
\item[-] For any $\sigma\in \Gal(H/K)$ and $x\in E(H)/{2 E(H)}$, define
  \[\sigma(x)=\phi(x^\sigma)\]
where $\phi\in Hom(E^\sigma,E)$ is chosen so that $\phi$ maps to $1$
under the isomorphism in Lemma~\ref{Lem3}

\item[-] For any $\sigma\in \Gal(H/K)$ and $x\in \Sel_{2}(E/H)$, define
  \[\sigma(x)=\phi(x^\sigma)\]
where $\phi\in \Hom(E^\sigma,E)$ is chosen so that $\phi$ maps to
$1$ under the isomorphism in Lemma~\ref{Lem3}

\item[-] For any $\sigma\in \Gal(H/K)$ and $x\in \Sha(E/H)[2]$, define
\[\sigma(x)=\phi(x^\sigma)\]
where $\phi\in Hom(E^\sigma,E)$ is chosen so that $\phi$ maps to $1$
under the isomorphism in Lemma~\ref{Lem3}.
\end{enumerate} It is
easy to verify the above actions are independent of the choose of
$\phi$.

\begin{prop}\label{prop1}The exact sequence
$$
\xymatrix@C=0.5cm{
  0 \ar[r] & E(H)/{2 E(H)} \ar[rr]^{\delta} && \Sel_{2}(E/H) \ar[rr]^{} && \Sha(E/H)[2] \ar[r] & 0 }
$$ is an exact sequence of $\Gal(H/Q)$ modules.
\end{prop}
\begin{proof}It is enough to show $\delta$ is a homomorphism of
$\Gal(H/Q)$-modules.

For any $P\in E(H)/{2 E(H)}$, assume $[2]Q=P$, then
$\delta(P)(g)=Q^g-Q$, for any $g\in G_H$. Choose $\phi\in
Hom(E^\sigma,E)$ such that $\phi\equiv1(2)$, then we have by
definition $\sigma(P)=\phi(P^{\sigma})$, so
  \[\begin{split}&[\delta(\sigma(P))](g)=[\delta(\phi(P^{\sigma}))](g)\\
                =&g(\phi(Q^{\sigma})-\phi(Q^{\sigma}))=\phi(g(Q^{\sigma})-\phi(Q^{\sigma}))\\
                =&\phi\circ{\sigma[\sigma^{-1}g\sigma(Q)-Q]}=\sigma[\delta(P)(g)],\end{split}\]
the proposition then follows.
\end{proof}

\subsection{Gross curves}
Let $p$ be a rational prime with $p>3$ and $p\equiv3\pmod4$.
Let $K=\BQ(\sqrt{-p})$, $\mathcal {O}$ the integer ring of $K$, $H=H_K$ be the
Hilbert class field of $K$. For any ideal
$a\subseteq\mathcal {O}$, let $K(a)$ be the
ray class field modulo $a$.

Consider the continuous homomorphism
$\phi_0:K^{\times}(\prod_{v}\mathcal {O}^{\times}_v)\rightarrow
K^{\times}$ ($\mathcal {O}^{\times}_{\infty}:=C^{\times}$)
satisfying

$(1)$$\phi_0|_{K^{\times}}=$id$_{K^{\times}}$

$(2)$
\[\xymatrix{
  \prod_{v}\CO^{\times}_v\ar[r]\ar[d]& \{\pm1\}\\
  \prod_{v|(\sqrt{-p})}\CO^{\times}_v\ar[r]& (\mathcal
  {O}/(\sqrt{-p}))^{\times}\ar[u]^\delta&
  }\]

Here $\delta$ maps $x=a+b\frac{1+\sqrt{-p}}{2}(a,b\in\BZ)$ to
$(\frac{x\pmod{\sqrt{-p}}}{p})=(\frac{a+\frac{b}{2}}{p})$, where
$(\frac{\cdot}{p})$ is the Jacobi symbol. Note that $p\equiv3\pmod4$
ensures this $\phi_0$ is well defined.

From
\[\xymatrix{
  0 \ar[r]& K^{\times}(\prod_{v}\mathcal
{O}^{\times}_v)\ar[r]& A_{K^{\times}}\ar[r]&Cl(K)\ar[r]&0
  }\]

we get
\[\xymatrix{
  0 \ar[r]& \Hom(CL(K),\bar{K}^{\times})\ar[r]&
\Hom(A_{K^{\times}},\bar{K}^{\times})\ar[r]&\Hom(K^{\times}(\prod_{v}\mathcal
{O}^{\times}_v),\bar{K}^{\times})\ar[r]&0
  }\]

because Ext$^1(Cl(K),\bar{K}^{\times})=0$ as $\bar{K}^{\times}$ is
divisible hence injective. From this, we have
\begin{thm}\label{construction}
There is a continuous homomorphism $\phi:\BA_{K^{\times}}\rightarrow
\bar{K}^{\times}$ such that $\phi|_{K^{\times}(\prod_{v}\mathcal
{O}^{\times}_v)}=\phi_0$; in particular this character is of
conductor $(\sqrt{-p})$. This character is unique up to
$\widehat{Cl(K)}.$
\end{thm}

Let $\chi:\BA_{H}^{\times}\rightarrow {K}^{\times}$ be defined as
$\chi=\phi\circ N^{H}_K$ where $N^{H}_K$ is the norm map.

By the CM theory, there is a unique isogeny class of elliptic curves
over $H$ with CM $\mathcal {O}$ and the associated character $\chi$. We will call any elliptic curves in this isogeny class a Gross curve of level $p$.

Here are the basic properties of the Gross curves (\cite{G3}):
\begin{thm}\label{summary}
Let $E$ be a Gross curve and $F=\BQ(j(E))$, then we have

(1) $E(F)_{tor}\simeq \BZ/2\BZ$ or $0$, according to whether $(\frac{2}{p})=1$ or $-1$;

(2) The $\epsilon$-factor of $L(E/F,s)$ equals to $(\frac{2}{p})$.
\end{thm}

\subsection{Computation of the 2-Selmer group}

In this subsection, we assume $p\equiv7\pmod8$. We will use the method in section2.1 to compute the rank of some quadratic twists of the Gross curve. Note that in $K$, we have $(2)=\omega\bar{\omega}$ with
$\omega=(\frac{1+\pi}{2},2)$ and $\mathcal{O}^\times=\{\pm1\}$.

In \cite{G}, Gross established the following results.

\begin{prop}\label{prop2}
Notations as above, then for any Gross curve $E$, we have
  \[E^{(d)}(H)/2E^{(d)}(H)\cong\mathcal{O}/2\mathcal{O}\bigoplus(\mathcal{O}/2\mathcal{O}[\Gal(H/K)])^{n(d)}\]
with some integer $n(d)$ (so that $n(d)\cdot
h_K=rank_{\mathcal{O}}E^{(d)}(H)$).

In particular, we have
$n(d)+1\leq
rank_{\mathcal{O}/2\mathcal{O}}(\Sel_2(E^{(d)}/H))^{\Gal(H/K)}$.
\end{prop}

\begin{lem}\label{Lem9}
For any two Gross curves $E_1,E_2$, we have $\Sel_2(E_1/H)\cong \Sel_2(E_2/H)$
as $\Gal(H/K)$-modules.
\end{lem}
\begin{proof}As $E_i[2]\subseteq E_i(H)$, we have
$H^1(G_H,E_i[2])\cong\Hom(G_H,E_i[2])$.

The density theorem implies there is an $\bar{H}$-isogeny
$\phi:E_1\rightarrow E_2$ such that $deg(\phi)$ is odd. But $E_1$
and $E_2$ are $H$-isogenous, so by Lemma 2, we have $\phi$ is an
$H$-isogeny.

This $\phi$ induces a group isomorphism (also denoted by $\phi$)
$\phi:\Hom(G_H,E_1[2])\rightarrow \Hom(G_H,E_2[2])$, sending
$\psi\in \Hom(G_H,E_1[2])$ to $\phi\circ\psi$. We know from
Proposition~\ref{prop2} that $E_i[2]$ are trivial
$\Gal(H/K)$-modules. So for any $\psi\in \Hom(G_H,E_1[2])$, $g\in
G_H$ and $\sigma\in \Gal(H/K)$, we have
$(\phi\circ\psi)^{\sigma}(g)={\sigma}(\phi\circ\psi({\sigma}^{-1}g{\sigma}))
=\phi\circ\psi({\sigma}^{-1}g{\sigma})=\phi(\sigma(\psi({\sigma}^{-1}g{\sigma})))
=\phi(\psi^{\sigma}(g))=\phi\circ\psi^{\sigma}(g)$, i.e. $\psi$ is a
homomorphism of $\Gal(H/K)$-modules. And this gives the desired
homomorphism between $\Sel_2(E_1/H)$ and $\Sel_2(E_2/H)$.
\end{proof}

In the following, we write $S^{(d)}$ for
$(\Sel_2(E^{(d)}/H))^{\Gal(H/K)}$ (for any $E\in[C]$) .
By Lemma~\ref{Lem9}, we only need to compute $S^{(d)}$ for any fixed $E\in[C]$. But we have the
following
\begin{lem}\label{Lem4}
There is a unique Gross curve $E(p)$ such that $\Delta(E(p)/F)=(-p^3)$, where $F=\BQ(j(E))$.
\end{lem}
\begin{proof} C.f. \cite{G}, Theorem 12.2.1.
\end{proof}

We will do the computation for this $E(p)$.

Recall that from
$$\xymatrix@C=0.5cm{
  0 \ar[r] & E^{(d)}(p)(H)[2] \ar[rr]^{} && E^{(d)}(p)(\bar{H}) \ar[rr]^{2} && E^{(d)}(p)(\bar{H}) \ar[r] & 0
  }
$$
we get the following diagram£º
\[\xymatrix{
  0 \ar[r]&\frac{E^{(d)}(p)(H)}{2 E^{(d)}(p)}         \ar[r]^\delta\ar[d]&H^1(G_H,E^{(d)}(p)[2])\ar[r]\ar[d]    &H^1(G_H,E^{(d)}(p))[2] \ar[r] &0 \\
  0 \ar[r]&\prod\frac{E^{(d)}(p)(H_v)}{2 E^{(d)}(p)(H_v)}\ar[r]^{\delta_v}&\prod H^1(G_{H_v},E^{(d)}(p)[2])\ar[r]&\prod H^1(G_{H_v},E^{(d)}(p))[2] \ar[r] &0
}\]

\begin{lem}\label{Lem5}
There is a basis of $E^{(d)}(p)[2]$, such that
  \[im(\delta_v)=\{(x,y)\in H^{\times}_v/{H^{\times}_v}^2:x\in1+{\omega}^2\mathcal{O}^{\times}_v,y\in1+{\omega}\mathcal{O}^{\times}_v\}\]
for all place $v$ of $H$ over $\omega$, and
  \[im(\delta_w)=\{(x,y)\in H^{\times}_w/{H^{\times}_w}^2:x\in1+\bar{\omega}\mathcal{O}^{\times}_w,y\in1+\bar{\omega}^2\mathcal{O}^{\times}_w\}\]
for all place $w$ of $H$ over $\bar{\omega}$.
\end{lem}
\begin{proof}
First, we show $E(p)$ has ordinary reduction at every place $v$ of $H$ over $\omega$ and the same for places over $\bar{\omega}$.
Suppose $N(\wp_v)=(\wp_\omega)^{f(v/\omega)}=(f_\omega)$, then $\chi_p(\pi_v)=\pm f_\omega$. We need to show that $a_v=\pm(f_\omega+\bar{f_\omega})$ is odd. Assume $f_\omega=x+y\cdot\frac{1+\pi}{2}$ with $a,b\in\BZ$, then $a_v=\pm(y+2a)$. As $Nf_\omega=f_\omega\cdot \bar{f_\omega}=(N\wp_{\omega})^{f(v/\omega)}=2^{f(v/\omega)}$, we have $x^2+xy+y^2\cdot\frac{1+p}{4}=2^{f(v/\omega)}$ which is even. If $2\mid y$, then $2\nmid x$ because $2\nmid f_\omega$, then $x^2+xy+y^2\cdot\frac{1+p}{4}$ is odd which is a contradiction, so $2\nmid y$ and $a_v$ is odd. Now it follows that $E(p)^{(d)}$ has good ordinary reduction because its character differs from $\chi_p$ by a quadratic character unramified over $2$.

From \cite{BK}, Lemma 3.5, there is a unique two torsion point $P_1$ such that $P_1\equiv\mathcal{O}\pmod\omega$. Because $\Delta(E(p)/F)=(-p^3)$ is odd, $P_1$ can not belongs to $E(p)(F)$ for otherwise $P_1$ will be 2-integral which contradicts to $P_1\equiv\mathcal{O}\pmod\omega$. Let $P_2=\bar{P_1}$, then $P_2$ is the unique two torsion such that $P_2\equiv\mathcal{O}\pmod{\bar{\omega}}$. The assertion follows from \cite{BK}, Proposition 3.6 by taking $P_1,P_2$ as the basis.
\end{proof}

To state our results, we introduce the following notations.

Let $d=\prod^n_{i=1}(-q_i)\cdot\prod^m_{j=1}(q'_j)\cdot\prod^l_{k=1}(Q^*_k)$ be an integer congruent to $1$ modulo $4$, where $q_i,q'_j$ are primes split in $K$ and $Q_k$ are primes inertia in $K$ with $q_i\equiv3\pmod4$ and $q'_j\equiv1\pmod4$.

Let $h=h_K$, then $q^h_i=f_i\cdot\bar{f_i}$ and ${q_j'}^h=g_j\cdot\bar{g_j}$ with $f_i,g_j\in\CO_K$.

\begin{lem}\label{Lem10} Notation as above, we may assume $f_i=a_i+b_i\frac{1+\pi}{2}$ with $a_i\equiv1\pmod4$ and $v_2(b_i)=1$; $g_j=a'_j+b'_j\frac{1+\pi}{2}$ with $a'_j\equiv1\pmod4$ and $v_2(b'_j)\geq2$.
\end{lem}
\begin{proof}
Write $f_i=a_i+b_i\frac{1+\pi}{2}$, then $a^2_i+a_ib_i+b^2_i\frac{1+p}{4}=q^h_i$ is odd, so it is easy to see that $2\nmid a_i$ but $2\mid b_i$, and hence we can multiply it by $\pm1$ so that $a_i\equiv1\pmod4$.

Now we have $f_i-1\equiv0\pmod{2\omega}$ and $\bar{f_i}-1\equiv0\pmod{2\bar{\omega}}$, so $(f_i-1)(\bar{f_i}-1)\equiv0\pmod4$. On the other hand, $f_i\cdot\bar{f_i}-1=q^h_i-1\equiv2\pmod4$ because $h$ is odd by the genus theory, so $(f_i-1)+(\bar{f_i}-1)\equiv2\pmod4$, i.e. $b_i+2(a_i-1)\equiv2\pmod4$, then we have $b_i\equiv2\pmod4$.

The proof for the second assertion is similar.
\end{proof}

\begin{lem}\label{Lem6} Let $d\equiv1\pmod4$ be an integer and notations as above,
then we have

(i) $E^{(d)}(p)$ has good reduction at all the places not dividing
$pd$;

(ii) There is a basis of $E^{(d)}(p)[2]$, such that
  \[S^{(d)}\subseteq H_d:=\{(\alpha,\beta)\in({K^\times}/K^{\times2})^2|\alpha=(-\pi)^a\prod^n_{i=1}
[f^{s_i}_i\cdot(-\bar{f_i}^{t_i})]\cdot\prod^m_{j=1}[g^{r_j}_j\cdot\bar{g_j}^{u_j}]\cdot\prod^l_{k=1}[(Q^*_k)^{v_k}],$$

$$\beta=(\pi)^{a'}\prod^n_{i=1}
[(-f^{s'_i}_i)\cdot\bar{f_i}^{t'_i}]\cdot\prod^m_{j=1}[g^{r'_j}_j\cdot\bar{g_j}^{u'_j}]\cdot\prod^l_{k=1}[(Q^*_k)^{v'_k}]\}\]
where $a,...,v_k'=0\ or\
1$;

(iii) For any $v\nmid 2$, we have
$\#im(\delta_v)=4$.

\end{lem}
\begin{proof}

(i) This is because $E(p)$ only has bad reduction at the places over
$p$ and $d$ is congruent to 1 mod 4;

(ii) Note that by the genus theory, the order of $\Gal(H/K)$ is odd,
so both $H^1(\Gal(H/K),E[2])$ and $H^1(\Gal(H/K),E[2])$ are zero.
Then by the Serre-Hoschild exact sequence, we have
$H^1(G_K,E[2])\cong H^1(G_H,E[2])^{\Gal(H/K)}$ and so
$S^{(d)}\subseteq H^1(G_K,E[2])$.

If $(\alpha,\beta)\in ({K^\times}/K^{\times2})^2$ belongs to
$S^{(d)}$, then by Lemma~\ref{Lem5} and (i) above, we have $\alpha,\beta\in
\mathcal{O}^\times_{H_v}$ modulo $(H^\times_v)^2$ (for any $v\nmid
pd$). But $H$ is unramified over $K$, so $\alpha,\beta\in
\mathcal{O}^\times_{K_w}$ modulo $(K^\times_w)^2$ (for any $w\nmid
pd$). Also because $h$ is odd, we find
$\alpha=(\alpha)^h=\pm(\pi)\prod^n_{i=1}
[f^{s_i}_i\cdot(\bar{f_i}^{t_i})]\cdot\prod^m_{j=1}[g^{r_j}_j\cdot\bar{g_j}^{u_j}]\cdot\prod^l_{k=1}[(Q^*_k)^{v_k}]$ with $a_i,...,v_k=0,1$ and similarly for $\beta$.

By Lemma~\ref{Lem5}, we can choose a basis of
$E(p)[2]$ such that $\alpha\equiv1\pmod{\omega^2}$ and $\beta\equiv1\pmod{\bar{\omega}}^2$, so we get the
result;

(iii) Suppose $v\nmid 2$. By the theory of formal groups, there is
$\mathrm{M}\subseteq E^{(d)}(p)(H_v)$ such that $\mathrm{M}\cong
\CO_v$ and $E^{(d)}(p)(H_v)/\mathrm{M}$ is finite. Consider

*******

Apply snake lemma, we get $|E^{(d)}(p)(H_v)/2A(H_v)|\cdot |\CO_v[2]|=|E^{(d)}(p)(H_v)[2]|\cdot |\CO_v/2\CO_v|$. But as $v\nmid 2$, then we have
$|\CO_v[2]|=|\CO_v/2\CO_v|=1$ and the
result follows;

(iv) Just by the definition of the Selmer group.
\end{proof}

From Lemma~\ref{Lem6} we know that to compute $S^{(d)}$, it is
necessary to know the image of $E^{(d)}(p)[2]$ under $\delta$. For
this, we have the following

\begin{lem}\label{Lem7}For any $d\in\BZ$,
there is a basis of $E^{(d)}(p)[2]$ such that
  \[\delta(E^{(d)}(p)[2])=\{(1,1),(-\pi d,1),(1,\pi d),(-\pi d,\pi d)\}.\]
And we have
  \[im(\delta_v)=\delta_v(E^{(d)}(p)[2])\]
for any $v\mid pd$.
\end{lem}
\begin{proof}
For the first assertion, it is enough to verify this for the case $d=1$. Fix the basis as in Lemma~\ref{Lem5}.

Take a Weierstrass equation over $H$ of
$E(p):y^2=(x-e_1)(x-e_2)(x-e_3)$ with $\Delta(E(p))=-p^3$. Since
$E(p)$ has potentially good reduction everywhere, we can find some
finite extension of $H$ such that $E(p)$ has good reduction at
$\pi$. Then a change of coordinates of the form
  \[\left\{\begin{array}{l}x=\pi X+r\\y=\pi^{\frac{3}{2}}Y+s\pi X+t\end{array}\right.\]
gives a Weierstrass equation $\mathcal {E}(p):f(X,Y)=0$ with good
reduction at $\pi$. Notice that $P_i=(e_i,0)$'s are the 2-torsion
points, we have
  \[v_{\pi}(X(P_i)-X(P_j))\geq0,\ \forall i\neq j,\]
and then $v_{\pi}(e_i-e_j)\geq 1$. But $\Delta(E(p))=-p^3$ implies
$2\sum_{i<j}v_{\pi}(e_i-e_j)=6$, hence we have $v_{\pi}(e_i-e_j)=1$.

By~\cite{Siv}, Proposition 14, we have
  \[\begin{split}\delta(P_0)&=(x_0,y_0)=(1,1),\\
                 \delta(P_1)&=(x_1,y_1)=(\frac{e_1-e_3}{e_1-e_2},e_1-e_2),\\
                 \delta(P_2)&=(x_2,y_2)=(e_2-e_1,\frac{e_2-e_3}{e_2-e_1}),\\
                 \delta(P_3)&=(x_3,y_3)=(e_3-e_1,e_3-e_2).
  \end{split}\]
Since Lemma~\ref{Lem6} implies that $x_i,y_i\equiv(-1)^a\pi^b$ with
$a,b=0\ or\ 1$, by combining the above results and Lemma~\ref{Lem5}, we have
  \[\begin{split}\delta(P_0)&=(x_0,y_0)=(1,1),\\
                 \delta(P_1)&=(x_1,y_1)=(1,\pi),\\
                 \delta(P_2)&=(x_2,y_2)=(-\pi,1),\\
                 \delta(P_3)&=(x_3,y_3)=(-\pi,\pi).
  \end{split}\]

For the second assertion, we note that the four elements
  \[(1,1),(-\pi d,1),(1,\pi d),(-\pi d,\pi d)\]
are distinct in $K^{\times}_v$ for any $v\mid pd$, so the result follows from Lemma~\ref{Lem6}, (iii).
\end{proof}

Now we can prove our main theorem which gives a complete description of the elements in $S^{(d)}$ for $d\equiv1\pmod4$.

\begin{thm}\label{main thm}
$(\alpha,\beta)\in
S^{(d)}$ is equivalent to $(\alpha,\beta)\in H_d$ and there is
\[(x_{i(v)},y_{i(v)})\in \{(1,1),(-\pi d,1),(1,\pi d),(-\pi d,\pi d)\}\]
such that $\alpha x_{i(v)}\in {K^\times_v}^2$ and $\beta y_{i(v)}\in {K^\times_v}^2$ for any place $v\mid pd$ of $K$.
\end{thm}
\begin{proof}
This follows from the definition of Selmer group, combining with Lemma~\ref{Lem6} and Lemma~\ref{Lem7}.
\end{proof}

In practice, one can always compute $S^{(d)}$ by Theorem~\ref{main thm}. In the following, we give a graphical description of it, which seems more convenient to use.

\begin{defn}
Let $d\equiv1\pmod4$ and $f_i,g_j,Q_k$ as above.

Define a (oriented) graph $G_d$ as following:

vertex of $G_d$=$\{-\pi,f_i,-\bar{f_i},g_j,\bar{g_j},Q^*_k\}_{1\leq i\leq n,1\leq j\leq m,1\leq k\leq l}$

arrows of $G_d$: there exist an arrow from $x$ to $y$ if and only if $(\frac{x}{y})=-1$ (here the symbol $(\frac{x}{y})$ is the quadratic residue symbol in $K$)
\end{defn}

\begin{thm}\label{main thm'}
For every $d\equiv1\pmod4$, we have
 \[\rank_{\mathcal{O}/2\mathcal{O}}(S^{(d)})=1+2t\]
where t is the number of even partitions of $G_d$.

In particular, $S^{(d)}$ is minimal if and only if $G_d$ is an odd graph.
\end{thm}

\begin{proof}

Define graph $G_d'$ with vertex $\{\pi,-f_i,\bar{f_i},g_j,\bar{g_j},Q^*_k\}_{1\leq i\leq n,1\leq j\leq m,1\leq k\leq l}$, and there is an arrow from $x$ to $y$ if and only if $(\frac{x}{y})=-1$.

Given $(\alpha,\beta)\in H_d$, we have a partition $V_{\alpha}\bigcup V_{n\alpha}$ of $G_d$ with $V_{\alpha}=\{x:x\mid\alpha\}$, and similarly a partition $V_{\beta}\bigcup V_{n\beta}$ of $G_d'$. Now $(\alpha,\beta)\in S^{(d)}$ means that $(\frac{\alpha}{x})=1$ for any $x\nmid\alpha$ and $(\frac{-\pi d/\alpha}{x})=1$ for any $x\mid\alpha$, and the same for $\beta$ which is equivalent to say that $V_{\alpha}\bigcup V_{n\alpha}$ and $V_{\beta}\bigcup V_{n\beta}$ are even partitions. Note that $\alpha$ and $-\pi d/\alpha$ correspond to the same partition and the same for $\beta$ and $\pi d/\beta$, we will obtain the assertion if we can show that if we can show the map $\phi:G\rightarrow G_d', -\pi\mapsto\pi,f_i\mapsto\bar{f_i},-\bar{f_i}\mapsto-f_i,g_j\mapsto\bar{g_j},\bar{g_j}\mapsto g_j,Q^*_k\mapsto Q^*_k$ is an isomorphism, i.e. there is an arrow from $x$ to $y$ if and only if there is an arrow from $\phi(x)$ to $\phi(y)$, which is obvious.
\end{proof}

\subsection{Numerical examples}
\begin{lem}\label{Lem8}
(i)If $Q$ is a prime such that $(\frac{-p}{Q})=-1$, then we have
$\pm\pi\in (K^\times_Q)^2$ if and only if $Q$ is congruent to $3$
modulo $4$;

(ii)If $Q$ is a prime such that $(\frac{-p}{Q})=-1$, then we always
have $-1\in (K^\times_Q)^2$;

(iii)If $Q_1$ and $Q_2$ are primes such that $(\frac{-p}{Q_i})=-1$
$(i=1,2)$, then we always have $Q^*_1\in (K^\times_{Q_2})^2$.

\end{lem}
\begin{proof}

(i)If $Q$ is congruent to $3$ modulo $4$. By Hensel lemma, it is
enough to solve $(a+b\pi)^2)\equiv\pm\pi$ $(Q)$. This is equivalent
to $a^2-pb^2\equiv0$ $(Q)$ and $2ab\equiv\pm1$ $(Q)$. So we only
need to show $a^4\equiv\frac{p}{4}$ $(Q)$ has solution in
$\mathbb{Z}$. But as $Q$ is congruent to $3$ modulo $4$, we have
$(\frac{p}{Q})=-(\frac{-p}{Q})=1$, so there is $x\in\mathbb{Z}$ such
that $x^2\equiv\frac{p}{4}$ $(Q)$. As one of $x$ and $-x$ is also a
square modulo $Q$, we can then get the solution of
$a^4\equiv\frac{p}{4}$ $(Q)$.

If $Q$ is congruent to $1$ modulo $4$, then
$(\frac{p}{Q})=(\frac{-p}{Q})=-1$, so the equation doesn't have any
solutions.

(ii)This is well known if $Q$ is congruent to $1$ modulo $4$. B ut
$(1)$ above implies this is also true for $Q$ congruent to $3$
modulo $4$.

(iii)By Hensel lemma, it is enough to solve $(a+b\pi)^2\equiv
Q^*_1$ $(Q_2)$. This is equivalent to $a^2-pb^2\equiv Q^*_1$ $(Q_2)$
and $2ab\equiv 0$ $(Q_2)$. If $(\frac{Q^*_1}{Q_2})=1$, then we can
get a solution by setting $b\equiv0$. If $(\frac{Q^*_1}{Q_2})=-1$,
then set $a\equiv0$ to solve $b^2\equiv -pQ^*_1$, which has solution
as $(\frac{-pQ^*_1}{Q_2})=1$.

\end{proof}

\begin{thm}\label{thmm}Let $d=\prod\limits^n_{i=1}Q^*_i$ be a square-free rational integer, where $Q_i$
are odd rational primes such that $(\frac{-p}{Q_i})=-1$, then
\[E^{(d)}(p)(H)=E^{(d)}(p)[2]\ and\ (\Sha(E^{(d)}(p)/H)[2])^{\Gal(H/K)}={1}\]
if and only if $Q_i\equiv1\pmod4$ for any $i=1,...,n$.

Moreover, we have $\rank_{\mathcal{O}/2\mathcal{O}}S^{(d)}\geq 1+k$, where $k$ is the number of those $Q_i$ which is congruent to $3$ module $4$.
\end{thm}

\begin{proof}
If all the $Q_i$ are congruent to $1$ module $4$, we want to show that $(\alpha,\beta)\in S^{(d)}$ implies
$(\alpha,\beta)\in im(E^{(d)}(p)[2])$.

Suppose there is some $(\alpha,\beta)\in S^{(d)}$ not in
$im(E^{(d)}(p)[2])$. then either $\alpha\neq1,-\pi$ or
$\beta\neq1,\pi$.

If $\alpha\neq1,-\pi$, multiplying suitable element in
$\beta\neq1,\pi$, we may assume $\pi\mid\alpha$. Then by Lemma 9, we
have $\alpha$ is not in ${K^\times_{Q_i}}^2$ for any
$Q_i\nmid\alpha$. So we must have $\alpha=1\ or\ -\pi$.

If $\beta\neq1,\pi$, multiplying suitable element in
$\beta\neq1,\pi$, we may assume $\pi\mid\beta$. Then by Lemma 9, we
have $\beta$ is not in ${K^\times_{Q_i}}^2$ for any $Q_i\nmid\beta$.
So we must have $\beta=1\ or\ \pi$.

If there is some $Q_i\equiv1\pmod4$, we claim that $(1,Q^*_i)\in S^{(d)}$. At $Q_j$ for $j\neq i$, we have $Q^*_i\in (K^\times_{Q_j})^2$; at $\pi$, as $\frac{Q^*_i}{p}=\frac{p}{Q_i}=1$, we also have $(1,Q^*_i)\in im(\delta_{\pi})$; at $Q_i$, multiply it by $(1,\pi d)$ to get $(1,\pi\prod_{j\neq i}Q^*_i)$ with $\pi\prod_{j\neq i}Q^*_i\in {K^{\times}_{Q_i}}^2$ by Lemma~\ref{Lem8}. Now the claim follows from Lemma~\ref{Lem6}, (iv). This complete the first assertion of Theorem~\ref{thmm}.

By the above, we see that we always have $(1,Q^*_i)\in S^{(d)}$ for $Q_i\equiv3\mod4$. Since these elements are linearly independent in $S^{d}$, we complete the proof of Theorem~\ref{thmm}.

\end{proof}

\begin{cor}
Let $d$ be as in Theorem~\ref{thmm}  with $d>0$ and $p>4d^2\lg|d|$, then the BSD conjecture is true for $E^{(d)}(p)$ and $\Sha(E^{(d)}(p)/H)[2]=\Sel_2(E^{(d)}(p)/H)$. In particular, we can construct arbitrarily large Shafarevich-Tate group by choosing $p$ large enough and $d$ contains enough $Q$ which is congruent to $3$ modulo $4$.
\end{cor}
\begin{proof}

Under the assumptions on $d$, we have $L(E^{(d)}(p)/H,1)\neq0$ by the main theorem of \cite{Y}. So by the Coates-Wiles theorem, we know that $E^{(d)}(p)(H)=E^{(d)}(p)[2]$, and the assertions follows immediately from Theorem~\ref{thmm}.
\end{proof}

\begin{lem}\label{Lem11}
(i)If $q\equiv3\pmod4$ and splits in $K$, $q=f\cdot\bar{f}$ with $f$ as in Lemma~\ref{Lem10}, then $(\frac{f}{\pi})(\frac{\pi}{f})=1$, $(\frac{\bar{f}}{\pi})(\frac{\pi}{\bar{f}})=-1$ and $(\frac{f}{\pi})=(\frac{\bar{f}}{f})$;

(ii)If $q\equiv1\pmod4$ and splits in $K$, $q=f\cdot\bar{f}$ with $f$ as in Lemma~\ref{Lem10}, then $(\frac{f}{\pi})(\frac{\pi}{f})=1$, $(\frac{\bar{f}}{\pi})(\frac{\pi}{\bar{f}})=1$ and $(\frac{f}{\pi})=(\frac{\bar{f}}{f})$.
\end{lem}

\begin{proof}
(i)Write $f=a+b\frac{1+\pi}{2}$, then $a\equiv1\mod4$ and $v_2(b)=1$ as in Lemma~\ref{Lem10}.

By \cite{N}, P415, Theorem(8.3), we have $(\frac{f}{\pi})(\frac{\pi}{f})=(\frac{f,\pi}{\omega})(\frac{f,\pi}{\bar{\omega}})$. But as $f\equiv1\pmod{\omega^2}$ and $\pi=1-2\frac{1-\pi}{2}\equiv1\pmod{\bar{\omega}^2}$, by \cite{S}, Chapter3, Theorem1, we deduce that $(\frac{f,\pi}{\omega})=(\frac{f,\pi}{\bar{\omega}})=1$, hence $(\frac{f}{\pi})(\frac{\pi}{f})=1$.

Because both $\bar{f}$ and $\pi$ are congruent to $-1$ modulo $\omega^2$, and $\pi\equiv1\pmod{\bar{\omega}^2}$, we have $(\frac{\bar{f},\pi}{\omega})=-1$ and $(\frac{\bar{f},\pi}{\bar{\omega}})=1$, hence $(\frac{\bar{f}}{\pi})(\frac{\pi}{\bar{f}})=-1$.

Now we show $(\frac{f}{\pi})=(\frac{\bar{f}}{f})$. As $(\frac{f}{\pi})(\frac{\pi}{f})=1$, we only need to show that $(\frac{\pi\bar{f}}{f})=(\frac{-b}{q})=1$. Since $v_2(b)=1$, we have $(\frac{-b}{q})=(\frac{2}{q})(\frac{-b/2}{q})=(\frac{2}{q})(\frac{q}{(-b/2)})(\frac{2}{q})$.

Because $a^2+ab+b^2\frac{p+1}{4}=q^h$ and $h$ is odd, we have $(\frac{q}{(-b/2)})=1$.

Because $2\mid b$, we have $a^2+ab\equiv1+ab\equiv q^h\pmod8$. Then if $q\equiv3\pmod8$, we have $b\equiv2\pmod8$; if $q\equiv7\pmod8$, we have $b\equiv6\pmod8$, so that $(\frac{2}{q})(\frac{2}{q})=1$ always holds. This finishes the proof of (i).

(ii) The proof of the first two assertions are similar to the proof in (i), and we show $(\frac{f}{\pi})=(\frac{\bar{f}}{f})$, or equivalently, $(\frac{-b}{q})=1$.

Assume $|-b|=2^ec$ with $c$ odd, then $(\frac{-b}{q})=(\frac{2}{q})^e(\frac{q}{c})$.

Because $a^2+ab+b^2\frac{p+1}{4}=q^h$ and $h$ is odd, we have $(\frac{q}{c})=1$.

If $q\equiv1\pmod8$, then $(\frac{-b}{q})=(\frac{2}{q})^e=1$. If $q\equiv5\pmod8$, then $1+ab\equiv5\pmod8$, hence $b\equiv4\pmod8$, i.e. $e=2$. So we also have $(\frac{-b}{q})=1$.
\end{proof}

\begin{prop}\label{prop1}
If $q\equiv3\pmod4$ splits in $K$, then
 \[\rank_{\mathcal{O}/2\mathcal{O}}S^{(q^*)}=3\]
\end{prop}
\begin{proof}
Notations as above.

If $(\frac{f}{\pi})=1$, then it's easy to verify by the above Lemma and Theorem~\ref{main thm} that $(f,1)$ and $(1,\bar{f})$ generate $S^{(q^*)}$. On the other hand, if $(\frac{f}{\pi})=-1$, then is is generated by $(-\bar{f},1)$ and $(1,-f)$.
\end{proof}

\begin{prop}
Suppose $d=\prod_{j=1}^mq'_j$ with $q'_j\equiv1\pmod4$ split in $K$, and $g_j$ as in Lemma~\ref{Lem10}. If $(\frac{g_j}{\pi})=-1$ for any $j$ and $(\frac{g_k}{g_j})=(\frac{\bar{g_k}}{g_j})=-1$, then
 \[E^{(d)}(p)(H)=E^{(d)}(p)[2]\ and\ (\Sha(E^{(d)}(p)/H)[2])^{\Gal(H/K)}={1}\].
\end{prop}
\begin{proof}
As $g_j-1\equiv1\pmod4$ and hence $\bar{g}_j-1\equiv1\pmod4$, we have $(\frac{x}{y})=(\frac{y}{x})$ for any $x,y\in G$, i.e. $G$ is an unoriented grapha. The hypothesis implies that there is an arrow between any two vertexes of $G$, and since there are odd number of vertexes, we know $G$ is an odd graph.
\end{proof}

\section{Heegner points on Eisenstein quotients}
\subsection{Eisenstein quotients}
In this section, let $X=X_0(N)/\BQ$ be the modular curve of level some positive integer $N$, $J=J_0(N)/\BQ$ be its Jacobian and $\BT=\BZ[\{T_\ell\}_\ell]\subseteq End(J/\BQ)$ be the (full) Hecke algebra of level $N$.

Recall that as  Riemann surfaces, we have $X(\BC)=\Gamma_0(N)\setminus \fH^*$, where $\fH^*=\fH\bigcup\BP^{1}(\BQ)$ and $\fH$ is the upper-half plane. The points of $S=\Gamma_0(N)\setminus\BP^{1}(\BQ)$ are called cusps of $X$ and are known to be rational over $\BQ(\mu_N)$. Take $i:X\rightarrow J$ to be the natural morphism which sending $x$ to $[x]=(x)-(\infty)$. This morphism is defined over $\BQ$ because $(\infty)$ is $\BQ$-rational, hence $i$ induces a homomorphism of $G_{\BQ}$-modules (also denoted by $i$) $i: Div^0(X)\rightarrow J(\bar{\BQ})$. We define the $cuspidal\ subgroup$ of $J$ to be the image of $Div^0(S)$ under $i$ and denote it by $C$. Then $C$ has a structure of $\BT[G_{\BQ}]$-module because the action of $\BT$ preserves cusps. More over, as we know that $C$ is a finite group, we can also view $C$ as a finite group scheme of $J$.

\begin{defn}
Supoose $P\in C$ is Hecke-eigen, that is to say the subgroup $\BZ\cdot P$ of $C$ is stable under $\BT$. Then we define $\BI(P)$ to be the ideal of $\BT$ annihilates $P$ and we shall call it the Eisenstein ideal corresponding to $P$. So if $P$ is of order $n$, then we we have an isomorphism $\BT/\BI(P)\simeq\BZ/n\BZ$ and we define $m_q=(q,\BI(P))$ for any prime divisor $q$ of $n$.

For any $q\mid n$, let $m_q=(q,\BI)$ . Define the Eisenstein quotient corresponding to $P$ to be
\[\widetilde{J}(P)=J/(\bigcap_{k\geq0}\BI^k)J\]
and the $q$-Eisenstein quotient corresponding to $P$ to be
\[\widetilde{J}^{(q)}(P)=J/(\bigcap_{k\geq0}{m_q}^k)J\]
for any $q\mid n$.
\end{defn}

Now assume $K$ to be an imaginary quadratic field in which all the prime divisor of $N$ splits (i.e. $(K,N)$ satisfies the Heegner hypothesis). Then for any integer $c$ prime to $N$, there exits an ideal $\fN_c$ in $\CO_c$ such that $\CO_c/\fN_c\cong \BZ/N\BZ$. Let
\[x_c=\BC/\CO_c\rightarrow\BC/\fN^{-1}_c\in X(H_c)\]
where $H_c$ is the ray class field of $K$. Hence we construct a point $[x_c]$ in $J(H_c)$. For any character $\chi:Gal(H_c/K)\rightarrow\BC^{\times}$, define
\[y_{\chi}=\sum_{\sigma\in Gal(H_c/K)}\chi^{-1}(\sigma)[x^{\sigma}_c]\]
in $J(H_c)\otimes\BC)^{\chi}$. When $c=1$ and $\chi$ is trivial, we denote the corresponding point by $y_K$. The finite dimensional vector space $J(H_c)\otimes\BC$ also admit a natural action by $\BT$, which commutes with the action of $G_Q$. For any homomorphism of algebras $f:\BT\rightarrow\BC$, define $y_{\chi,f}$ to be the projection of $y_{\chi}$ on $(J(H_c)\otimes\BC)^{\chi,f}$. The basic question is to determine whether these point $y_{\chi,f}$ is zero.

In the following subsection, we introduce a method (due to Gross) to test whether the projection of the Heegner points on the Eisenstein quotients are non-torsion.

\subsection{Eisenstein descent}
The canonical morphism $i$ induces an isomorphism $i^*:\hat{J}\simeq J$. Let $P$ be a cuspidal point of order $n$ defined over some field $M(\subseteq\BQ(\mu_N))$, and $D$ a cuspidal divisor representing $P$. Let $P'$ be the inverse image of $P$ under $i^*$.

Then $P'$ gives a morphism of $G_M$-modules $J[n]\rightarrow\mu_n$ via the Weil pairings. Combining this with the Kummer map $J(F)/nJ(F)\rightarrow H^1(F,J[n])$ where $F$ is some number field containing $M$, we get a map of $G_F$-modules
\[\delta(P):J(F)/nJ(F)\rightarrow F^{\times}\otimes\BZ/n\BZ\]
We call this the Eisenstein descent corresponding to $P$ over $F$.

\begin{prop}\label{modular unit}
$\delta(P)(\sum n_i(x_i))=\prod f^{n_i}(x_i)$, where $f$ is the modular unit such that $div(f)=nD$.
\end{prop}
\begin{proof}
Let $D'$ be the divisor on $J$ representing $P'$. Then both $nD'$ and $[n]^*D'$ are principle, say $nD'=div(F)$ and $[n]^*D'=div(G)$. By \cite{Mumford},P184, Lemma, $<Q,P'>=G(x)/G(x+Q)$ for any $Q\in J[n]$.
\end{proof}

\begin{prop}
If $P$ is $\BT$-eigen and $T_{\ell}P=T^*_{\ell}P$ for any $\ell\mid N$, then $\delta(P)$ is a morphism of $\BT$-modules.
\end{prop}
\begin{proof}
Let $\alpha_\ell,\ \beta_\ell:X_0(\ell p^2)\rightarrow X$ be the two morphisms sending $(E,C,D)$ ($C$ the $\ell$-part and $D$ the $p^2$-part) to  $(E,D)$ and $(E/C,(C+D)/D)$ respectively, then $T_\ell$ is by definition $\circ \beta_{\ell,*}\alpha^*_{\ell}$.

By the construction, $f(T_\ell z)=f(\ell z)\prod^{\ell-1}_{i=0}f(\frac{z+i}{\ell})=\alpha_{\ell,*}\circ \beta^*_\ell(f)$. As $div(\alpha_{\ell,*}\circ \beta^*_\ell(f))=\alpha_{\ell,*}\circ \beta^*_\ell[div(f)]=T^*_\ell(div(f))=(\ell+1)\cdot div(f)$, so we are done.
\end{proof}

When $\delta(P)$ is a homomorphism of $\BT$-modules, we can further localize it and define
\[\delta(P)_q:J(F)/nJ(F)\bigotimes\BT_{m_q}\rightarrow F^{\times}\otimes\BZ_q/n\BZ_q\]
for any $q\mid n$

\subsection{$\eta$-quotient}
In this subsection, we will consider the Eisenstein descent in the case that $P$ is given by a rational cuspidal divisor.

Let $\eta(z)$ be the Dedkind $\eta$-function and let $\eta_d(z):=\eta(dz)$ for any integer $d$. Let $N$ be a positive integer as before. For any family of integers $\mathrm{r}=(r_d)$ indexed by the positive divisors of $N$, define
\[g_{\mathrm{r}}=\prod_{d\mid N}\eta^{r_d}_d\]
we call any such function a Dedkind $\eta$-product. We have the following proposition
\begin{prop}\label{Dedkind}
$g_{\mathrm{r}}\in\BQ(X)$ if and only if the following four conditions are satisfied:

(1) $\sum_{d\mid N}r_d=0$;

(2) $\sum_{d\mid N}dr_d\equiv0\pmod{24}$;

(3) $\sum_{d\mid N}\frac{N}{d}r_d\equiv0\pmod{24}$;

(4) $\prod_{d\mid N}d^{r_d}\in \BQ^{\times2}$
\end{prop}
\begin{proof}
See \cite{Li}.
\end{proof}

As representatives of the cusps of $X$, we choose $\frac{x}{d}$ where $d$ is a positive divisor of $N$ and $(x,d)=1$ with $x$ taken modulo $(d,\frac{N}{d})$. We call such a cusp is of level $d$ and  it is defined over $\BQ(\mu_m)$ where $m=(d,\frac{N}{d})$. The cusps of level $d$ form an orbit under the action of $Gal(\BQ(\mu_m)/\BQ)$. Let $D_d$ be the rational divisor on $X$ defined as the sum of the cusps of level $d$ (each with multiplicity one) and $P_d$ the corresponding rational point on $J$, that is to say $P_d\in C(\BQ)$. We have the following proposition for the relation between the rational cuspidal divisor on $X$ and the Dedkind $\eta$-product.
\begin{prop}\label{eta quotient}
Let $D=\sum_{d\mid N}m_d\cdot D_d$ be a rational cuspidal divisor of degree $0$ on $X$ and $P$ the corresponding point in $J(\BQ)$, then there is a Dedkind $\eta$-product $g_{\mathrm{r}}\in \BQ(X)$ such that $nD=div(g_{\mathrm{r}})$ where $n$ is the order of $P$.
\end{prop}
\begin{proof}
See \cite{Li}.
\end{proof}

Let $K$ be an imaginary quadratic field in which all the prime divisor of $N$ splits, then there is an ideal $\fN$ in $K$ such that $\CO_K/\fN\simeq\BZ/N\BZ$. Let $y_K\in J(H_K)$ be the Heegner point. For each $d\mid N$, we denote by $\fN_d$ the unique ideal such that $\fN_d\mid \fN$ and $\CO_K/\fN_d\simeq\BZ/d\BZ$. For any Dedkind $eta$-product $g_{\mathrm{r}}\in\BQ(X)$ for some $\mathrm{r}=(r_d)_{d\mid N}$, we have
\[\prod_{d\mid N}\fN^{r_d}_d=\fa_{\mathrm{r}}^{-2}\]
by the condition $(4)$ of Proposition~\ref{Dedkind}, for some rational ideal $\fa_{\mathrm{r}}$ in $K$.

\begin{thm}\label{Eisenstein descent for rational divisor}
Let $D=\sum_{d\mid N}m_d\cdot D_d$ be a rational cuspidal divisor of degree $0$ on $X$ and $P$ the corresponding point in $J(\BQ)$, $g_{\mathrm{r}}\in \BQ(X)$ is the Dedkind $\eta$-product such that $nD=div(g_{\mathrm{r}})$ where $n$ is the order of $P$, then
\[\delta(P)(y_K-\overline{y_K})\equiv \zeta\cdot\alpha_{\mathrm{r}}^{h_{\mathrm{r}}}\pmod{K^{\times n}}\]
where $\zeta$ is a root of unit in $K$, $\alpha_{\mathrm{r}}$ is a generator of the ideal $\frac{\fa_{\mathrm{r}}}{\overline{\fa_{\mathrm{r}}}}^{o(\fa_{\mathrm{r}})}$ as above and $h_{\mathrm{r}}=\frac{h_K}{o(\fa_{\mathrm{r}})}$.

In particular, suppose $\delta(P)$ is a homomorphism of $\BT$-modules, $q$ a prime divisor of $n$ with $(q,6)=1$ and $\alpha_{\mathrm{r}}^{h_{\mathrm{r}}}$ is not zero in $K^{\times}\bigotimes\BZ_q/n\BZ_q$ and  $J[m_q](K)^{-}=0$, then the projection of $y_K$ on the $q$-Eisenstein quotient corresponding to $P$ is non-torsion.
\end{thm}
\begin{proof}
From the definition of the Dedkind $\eta$-quotient and the condition $(1)$ of Proposition~\ref{Dedkind}, we have
\[g_{\mathrm{r}}(y_K-\overline{y_K})^{24}=\prod_{\fa}\prod_{d}\frac{\Delta(\fN_d\fa)}{\Delta(\fa)}^{r_d}\cdot\frac{\Delta(\overline{\fN_d\fa)}}{\Delta(\overline{\fa})}^{r_d}\]

We know that $\frac{\Delta(\fa)}{\Delta(\fN_d\fa)}$ is an integral number in $H_K$ which generates the ideal $\fN^{12}_d$, hence prove the first claim.

For the second claim, as $q$ is prime to $6$, we can ignore the root of unit above. Then $\delta(P)_q(y_K-\overline{y_K})$ is not zero, and so $y_K-y_K$ is not zero in $J(K)^{-}\bigotimes\BT_{m_q}$. Hence $\delta(P)_q(y_K-\overline{y_K})$ is either non-torsion or $m_q$-torsion. But we assume that $J(K)^{-}[m_q]=0$, so we get the conclusion.
\end{proof}

\section{Prime level case}
In this section, we let $N$ to be a prime $p$.

\subsection{Eisenstein quotients}

Recall that the cusps of $X$ are $[0]$ and $[\infty]$ which are all rational, so $C$ is a cyclic group generated by $[0]-[\infty]$. In particular, it must be Hecke-eigen.

Let $\BT=\BZ[\{T_\ell\}_\ell]\subseteq End(J_0($p$))$ be the Hecke algebra of level $p$, here $\ell$ runs through all rational primes. Note that by the theorem of Atkin-Lehner, we have $w_p=-T_p$ in this prime level case, so our $\BT$ is just the one used in \cite{Mz}. Let $\BI=(T_p-1,\{T_\ell-(1+\ell)\}_{\ell\neq p})=(w_p+1,\{T_\ell-(1+\ell)\}_{\ell\neq p})$ be an ideal in $\BT$, then $\BI$ annihilates $C$ and we call $\BI$ the $Eisenstein\ ideal\ of\ level\ p$.

\begin{prop}\label{Mazur1}
Notations as above and let $n=\frac{p-1}{(12,p-1)}$, then:

$(1)$ The order of $C$ is $n$ and $C=J(\BQ)_{tor}$;

$(2)$ $\BT$ acts transitively on $C$ with kernel $\BI$;

$(3)$ Under the natural morphism $J\rightarrow\widetilde{J}$, we have $C\simeq\widetilde{J}(\BQ)_{tor}$.
\end{prop}
\begin{proof}
See \cite{Mz}, Theorem1.2 of Chapter3 and Theorem9.7 of Chapter2.
\end{proof}

So we have the Eisenstein quotient $\widetilde{J}$, and the $\widetilde{J}^{(q)}$ for each $q\mid n$

Let $f=(\frac{\Delta(z)}{\Delta(pz)})^{\frac{1}{m}}$ with $m=(p-1,12)$, then $div(f)=n((0)-(\infty))$. For any number field $F$, define
\[\delta_F:D'_0(F)\rightarrow F^{\times}\]
by the formula
\[\delta_F(\sum a_i(x_i))=\prod f(x_i)^{a_i}\]
where $D'_0(F)$ is the group of divisors of degree zero over $F$.

On principle divisors we find $\delta(div(g))=g(div(f))=(g(0)/g(\infty))^{n}$ by reciprocity. Hence $delta$ induces a homomorphism
\[\delta_F:J(F)\rightarrow F^{\times}\otimes\BZ/n\BZ\]
This map is called the Eisenstein descent (corresponding to $C$) over $F$.

Here is an explanation why call this a descent. Consider the Kummer map $J(F)/nJ(F)\rightarrow H^1(F,J[n])$. By the Weil pairing, the cuspidal point $[0]$ in $J$ of order $n$ gives a homomorphism of $G_\BQ$-modules $J[n]\rightarrow\mu_n$. Then the composition of these two maps is just the $\delta_F$ given above.

Viewing $\BZ/n\BZ$ as $\BT$-module by Proposition~\ref{Mazur1}, it is easy to check that $\delta$ is a homomorphism of $\BT$-modules. Then we have
\[\delta_{F,q}:J(F)\otimes\BT_{m_q}\rightarrow \BZ_q/n\BZ_q\]
for any $q\mid n$.

\subsection{Heegner points on $\widetilde{J}^{(q)}$ for odd $q$}

In this section, we summarize Gross' results.

\begin{prop}\label{Gross1}
Suppose $q\mid n$ and $(q,6)=1$. Let $K$ to be an imaginary quadratic field in which $p$ splits. If $v_p(h_K)<v_p(\frac{p-1}{(12,p-1)})$, then $y^{(q)}_K$ is of infinite order in $\widetilde{J}^{(q)}(K)$ for odd $q$ such that $(q,w_K)=1$, where $y^{(q)}_K$ to be the projection of $y_K$ to $\widetilde{J}^{(q)}$ .
\end{prop}
\begin{proof}
First we show that $\delta_K(y_K-\bar{y_K})\neq0$.

By definition, we have
\[\delta_{K,q}(y_K-\bar{y_K})=(\prod_{\fa\in Cl(\CO_K)}\frac{\Delta(\fa)}{\Delta(\fp\fa)}\frac{\Delta(\overline{\fp\fa)}}{\Delta(\overline{\fa})})^{\frac{1}{m}}\]
where $\fp$ is a prime ideal in $K$ over $p$.
As for any ideal $\fb\in Cl(\mathcal{O}_K)$, $\frac{\Delta(\fp\fa)}{\Delta(\fa)}$ is a number in $H_K$ generates $\fp^{-12}$, we find that \[\delta_{K,q}(y_K-\bar{y_K})=u\cdot\alpha^{\frac{12h}{m}}\pmod K^{\times n}\]
with $u\in \mathcal{O}^{\times}_K$ and $\alpha\in K^{\times}$ such that $(\alpha)=(\fp/\bar{\fp})^{o(\fp)}$.

Note that $\alpha$ is not a $q$-th power in $K$. This is because if $x^q=\alpha$, then $(x)=(\fp/\bar{\fp})^{o(\fp)/q}$. As $\bar{\fp}$ is the converse of $\fp$ in the ideal class group, we find that $o(\fp)|\frac{2 o(\fp)}{q}$, which is impossible as $q$ is odd.

So when $ord_q(h)<ord_q(n)$, we will have $\delta_{K,q}(y_K-\bar{y_K})\neq 0$, hence $y_k\neq 0$ in $J(K)^{-}\bigotimes\mathbb{T}_{m_q}$.

Mazur shows that $J[m_q]=\BZ/q\BZ\oplus\mu_q$, so that $J(K)^{-}[m_q]=0$, so all points in $J(K)^{-}\bigotimes\mathbb{T}_{m_q}$ is not torsion. This completes the proof.

\end{proof}

\subsection{Heegner points on $\widetilde{J}^{(2)}$}
\begin{example}\label{Example1}
Suppose $p$ is of the form $u^2+64$ for some (odd) integer $u$. Note that $2\mid n$ in this situation, so that the $2$-Eisenstein quotient $\widetilde{J}^{(2)}$ exists.

On the other hand, it is known that when $p$ is of the above form, there is a unique isogeny class of elliptic curves of conductor $p$ such that each curve in it has a point of order $2$ (\cite{NS}). These curves are called Neumann-Setzer curves.

If $E$ is a Neumann-Setzer curve, then $E$ is a factor of $\widetilde{J}^{(2)}$ (\cite{Mz}, Chapter3, Proposition7.4). Moreover, if $u=\pm3\pmod8$, then $\widetilde{J}^{(2)}$ is simple (\cite{Mz}, Chapter3, Proposition7.5), so is a Neumann-Setzer itself.
\end{example}

We recall the following lemma
\begin{lem}(Birch's Lemma)
Let $A$ be an abelian variety over $\BQ$ and $f:X_0(N)\rightarrow A$ be a morphism over $\BQ$. If $f+f^{w_N}$ is a constant which does not belong to $2\cdot A(\BQ)$, then the image of the Heegner point $y_K$ on $A$ is not torsion, here $w_N$ is the Atkin-Lehner involution.
\end{lem}
\begin{proof}
***
\end{proof}

For the $2$-Eisenstein quotient, we can prove the following
\begin{thm}\label{Heegner point on 2-Eisenstein}
Suppose $2\mid\frac{p-1}{(12,p-1)}$. If $h_K$ is odd, then $y^{(2)}_K$ is of infinite order in $\widetilde{J}^{(2)}(K)$.
\end{thm}
\begin{proof}
Consider the morphism $f:X\rightarrow\widetilde{J}^{(2)}$ obtained from the composition of the natural $X\rightarrow J$ and the projection $J\rightarrow\widetilde{J}^{(2)}$.

By the definition of the Eisenstein ideal $\BI$, we know that $w_p$ acts as $-1$ on $\widetilde{J}$ and hence also $-1$ on $\widetilde{J}^{(2)}$, so $f+f^{w_p}$ is a constant morphism.The image is just the projection of $[0]-[\infty]$.

By Theorem~\ref{Mazur2}, the projection of $[0]-[\infty]$ on $\widetilde{J}^{(2)}$ is a generator of $\widetilde{J}^{(2)}(\BQ)$ which is a cyclic $2$-group. In particular, the image of $f+f^{w_p}$ is not in $2\cdot\widetilde{J}^{(2)}(\BQ)$, so we get the conclusion by using Birch's Lemma.
\end{proof}

\begin{cor}\label{Heegner point on NS-curve}
Suppose $p$ is of the form $u^2+64$ with $u=\pm3\pmod8$ and let $E$ be a Neumann-Setzer curve. If $K$ is an imaginary quadratic field with odd class number such that $p$ splits in $K$, then $rank_{\BZ}(E(K))=1$ and $\Sha(E/K)$ is finite.
\end{cor}

$\\$
\begin{remark} When $p$ is inertia in $K$, we have the following results of Mazur and Gross (See \cite{Mz2}, Page231 and \cite{G2}):

(1) Suppose $q$ is an odd divisor of $\frac{p-1}{(12,p-1)}$ and $(q,K)\neq(3,\BQ(\sqrt{-3}))$. If $K$ is an imaginary quadratic field such that $p$ is inertia in $K$ and $q\nmid h_K$ where $h_K$ is the class number of $K$, then $\widetilde{J}^{(q)}(K)$ is finite and $\Sha(J/K)[m_q]=0$;

(2) Assume $2\mid\frac{p-1}{(12,p-1)}$. If $K$ is an imaginary quadratic field such that $p$ is inertia in $K$ and $h_K$ is odd, then there is a cusp form which is congruent to $\delta\pmod2$, such that $L(f,K,1)\neq0$, where $\delta$ is the Eisenstein series.

So if $p$ is of the form $u^2+64$ with $u=\pm3\pmod8$ and $K$ is an imaginary quadratic field such that $p$ is inertia in $K$ and $h_K$ is odd, then $E(K)$ is finite for any Neumann-Setzer curve $E$, because there is a non-trivial morphism $E\rightarrow\widetilde{J}^{(2)}$  and $\widetilde{J}^{(2)}$ is simple when $u=\pm3\pmod8$ as we mentioned in Example~\ref{Example1}.

\end{remark}

\section{Level $p^2$ case}
In this section, we let $N=p^2$ be the square of a prime $p$.
\subsection{Eisenstein quotients}
Let $p$ be prime and $X/\BQ$ the modular curve of level $p^2$.

The cusps of $X$ are $[0]$, $[\infty]$ and $\{[\frac{i}{p}]\}_{1\leq i\leq p-1}$ with $[0]$ and $[\infty]$ $\BQ$-rational and $\{[\frac{i}{p}]\}_{1\leq i\leq p-1}$ form one orbit of $G_\BQ$. So the rational cuspidal divisor subgroup of $J(\BQ)$ $C$ is generated by $C_1=[0]-[\infty]$ and $C_p=\sum_{i=1}^{p-1}[\frac{i}{p}]-(p-1)[\infty]$. We know that both $C_1$ and $C_p$ have order $n=\frac{p^2-1}{24}$, and $C\cong(\BZ/\frac{p-1}{(p-1,12)}\BZ)^{\oplus2}\bigoplus(\BZ/\frac{p+1}{(p+1,12)}\BZ)$ (see \cite{Ling}).

Let $\BT=\BZ[\{T_n\}]\subseteq End(J/\BQ)$ be the ring of Hecke algebra of level $p^2$.

From the definition of the Hecke actions, we have $T_p\cdot C_p=0$ and $(T_l-(1+l))\cdot C_p=0$ for any $l\neq p$. Let \[\BI=(T_p,\{T_l-(1+l)\}_{l\neq p})\] be an ideal in $\BT$, which will be called the $Eisenstein\ ideal\ of\ level\ p^2\ corresponding to\ C_p$.

By the above Lemma, we see that there is a surjective homomorphism $\BT/\BI\rightarrow {\BZ}/{n\BZ}$ given by the action of $\BT$ on $C_p$.

\begin{lem}\label{cyclic torsion}
There is an integer $m$ such that $\BZ/m\BZ\cong \BT/\BI$.
\end{lem}
\begin{proof}
It is clear that the natural $\BZ\rightarrow \BT/\BI$ is surjective.

If $\BZ\cong \BT/\BI$, then under the perfect pairing
\[\BT\times S_2(\Gamma_0(p^2),\BZ)\rightarrow\BZ\]
(\cite{R} for the notations and results), we find that $S_2(\Gamma_0(p^2),\BZ)[\BI]\cong \BZ$ which means there is a rational newform whose eigenvalues $a_l$ is $l+1$ for any $l\neq p$. But this will give an elliptic curve over $\BQ$ which is supersingular at all good places, which is impossible.

So there is an $m\in\BZ$ such that $\BZ/m\BZ\cong \BT/\BI$.
\end{proof}

Let $\delta(z)=\sum_{(m,p)=1}\sigma(m)q^m$ ($q=e^{2\pi iz}$).

\begin{lem}\label{delta}
$\delta(z)\in M_2(\Gamma_0(p^2,\BZ))[\BI]$.
\end{lem}
\begin{proof}
Let $e(z)=(1-p)-24\sum^{\infty}_{m=1}\sigma'(m)q^m$ ($q=e^{2\pi iz}$) which is in $M_2(\Gamma_0(p),\BZ)$ (\cite{Mz}, our $e$ is his $e'$), and $e^{(p)}(z)=e(pz)$. It is easy to see that $\delta=\frac{1}{24}(e^{p}-e)$, so it is in $M_2(\Gamma_0(p^2,\BZ))$. It is easy by definition that $\delta$ is annihilated by $\BI$.
\end{proof}

From (\cite{Mz}, Chapter2, section5) we know the expansion of $e$ at $0$ is

$e|[\left(
                                                                                  \begin{array}{cc}
                                                                                    0 & -1 \\
                                                                                    1 & 0 \\
                                                                                  \end{array}
                                                                                \right)
]_2=\frac{1}{p}(p-1)+24\sum\frac{\sigma'(m)}{p}q^{\frac{m}{p}}$.

We can use this to determine the Fourier expansion of $\delta$ at $[0]$ and $[\frac{i}{p}]$($1\leqslant i\leqslant p-1$):

$\bullet$ At $[\frac{i}{p}]$:
Take $u,v\in \BZ$ such that $ui-pv=1$, then we have

$\delta|[\left(
           \begin{array}{cc}
             i & v \\
             p & u \\
           \end{array}
         \right)
]_2=\frac{p^2-1}{24p}+....$
\\

$\bullet$ At $[0]$:

$\delta|[\left(
           \begin{array}{cc}
             0 & -1 \\
             1 & 0 \\
           \end{array}
         \right)
]_2=\frac{(p^2-1)(1-p)}{24p}+....$

\begin{lem}\label{delta as cusp form}
Let $d$ be an integer prime to $p$, if $\delta\in S_2(\Gamma_0(p^2),\BZ/d\BZ)$, then $d|n$.
\end{lem}
\begin{proof}
By definition, $\delta\in S_2(\Gamma_0(p^2),\BZ/d\BZ)$ if and only if there are $f\in S_2(\Gamma_0(p^2),\BZ)$ and $g\in M_2(\Gamma_0(p^2),\BZ)$ such that  $\delta=f+dg$.

Since $g\in M_2(\Gamma_0(p^2),\BZ)$, we know the expansion of $g$ at any cusps belong to $\BZ[\frac{1}{p},\zeta_{p^2}]$ by (\cite{K}, Chapter1, Cor1.6.2). The assertion follows from this and the discussion of the expansions of $\delta$ above.
\end{proof}

\begin{thm}\label{index}
$(\BT/\BI)\otimes\BZ[\frac{1}{p}]\cong \BZ/n\BZ$.
\end{thm}
\begin{proof}
Let $m$ be as in Lemma~\ref{cyclic torsion}.

From the perfect pairing
\[\BT\times S_2(\Gamma_0(p^2),\BZ)\rightarrow\BZ\]
we get a perfect pairing
\[\BT/m\BT\times S_2(\Gamma_0(p^2),\BZ/m\BZ)\rightarrow\BZ/m\BZ\]

Then we have
\[\BT/\BI\times S_2(\Gamma_0(p^2),\BZ/m\BZ)[\BI]\rightarrow\BZ/m\BZ\]

or equivalently
\[\BT/\BI\cong Hom(S_2(\Gamma_0(p^2),\BZ/m\BZ)[\BI],\BZ/m\BZ)\]

and so
\[(\BT/\BI)\otimes\BZ[\frac{1}{p}]\cong Hom(S_2(\Gamma_0(p^2),\BZ[\frac{1}{p}]/m\BZ[\frac{1}{p}])[\BI],\BZ[\frac{1}{p}]/m\BZ[\frac{1}{p}])\]

By the $q$-expansion principle (see \cite{K}), $S_2(\Gamma_0(p^2),\BZ[\frac{1}{p}]/m\BZ[\frac{1}{p}])$ is generated by $c\delta$ for some $c|m$. Then we will have $c\delta=f+mg$ with some $f\in S_2(\Gamma_0(p^2),\BZ[\frac{1}{p}])$ and $g\in M_2(\Gamma_0(p^2),\BZ[\frac{1}{p}])$, then we get $\frac{1}{c}f=\delta-\frac{m}{c}g$ also in $S_2(\Gamma_0(p^2),\BZ[\frac{1}{p}])$. So $\delta\in S_2(\Gamma_0(p^2),\BZ[\frac{1}{p}]/\frac{m}{c}\BZ[\frac{1}{p}])$, and hence $\frac{m}{c}|n$ by Lemma~\ref{delta as cusp form}. So we see $\frac{m}{n}|c$, and hence $\#(S_2(\Gamma_0(p^2),\BZ[\frac{1}{p}]/mZ[\frac{1}{p}])[\BI])\leq n$. This completes the proof.
\end{proof}

As in the prime level case, let $m_q=(q,\BI)$ be a maximal ideal in $\BT$ for any $q|n$. Define the $Eisenstein\ quotient\ of\ level\ p^2\ corresponding\ to\ C_p$ to be
\[\widetilde{J}=J/(\bigcap_{k\geq0}\BI^k)J\]
and the $q$-$Eisenstein\ quotient\ of\ level\ p^2\ corresponding\ to\ C_p$ to be
\[\widetilde{J}^{(q)}=J/(\bigcap_{k\geq0}{m_q}^k)J\]
for any $q\mid n$.

Different from the prime level situation, the Atkin-Lehner involution $w_p$ is not equal to the Hecke operator $T_p$ when the level is $p^2$. So we further define
\[\widetilde{J}^{(q)}_{\pm}=(\widetilde{J}^{(q)})^{w_p\pm1=0}\]

\subsection{Structure of $J[m_q]$}
Let $S=Spec(\BZ[\frac{1}{p}])$ and $G/S$ any finite flat commutative group scheme over $S$. By (\cite{Mz}, Chapter1, P45-46), we have a natural bijection between \{flat closed subgroup schemes $H/S$ of $G/S$\} and \{sub-$G_{\BQ}$ modules of $G(\bar{\BQ})$\}.

An $S$-group scheme $G/S$ of order a power of $q$ ($q$ a rational prime) is called admissible if it is a finite flat group scheme over $S$ and has a filtration of flat closed subgroup schemes $0=G_0\subseteq G_1\subseteq....\subseteq G_n=G$ such that $G_i/G_{i-1}\cong (\BZ/q\BZ)/S$ or $\mu_q/S$. By the above remarks, this is equivalent to say that $G(\bar{\BQ})$ has a filtration of sub-representations with factors isomorphic to $\BZ/q\BZ$ or $\mu_q(\bar{\BQ})$.
\begin{prop}\label{admissible}
For any $q|n$, we have $J[m_q]$ is admissible.
\end{prop}
\begin{proof}
The proof is the same as in \cite{Mz}.
\end{proof}

Fix a prime $q|n$, then the rational point $C_p\in J(\BQ)[\BI]$ gives a rational point of $J[m_q]$, hence we have the exact sequence

\[\xymatrix{
  0 \ar[r]& \BZ/q\BZ \ar[r]&J[m_q]\ar[r]& V\ar[r]&0
  }\]

for some $V/S$ which is also admissible.

\begin{lem}\label{multiplicative type}
If $q$ is odd, then $V\cong (\mu_q)^{\bigoplus d}$ for some $d\geq 1$.
\end{lem}
\begin{proof}
First, we show $V$ is of multiplicative type, i.e. there's no embedding of $\BZ/q\BZ$ in $V$.

Suppose that there is an embedding $\BZ/q\BZ\rightarrow V$. Let $G$ be the pull back of this constant group scheme. Then there is an embedding of $(\BZ/q\BZ)^{\bigoplus2}$ in $J[m_q]$. By reduction to $\BF_q$, we get
\[dim_{\BF_q}(J/{\BF_q})[I](\bar{\BF_q})\geq2\]
which is impossible by the $q$-expansion principle as there is an injection $(J/{\BF_q})[I](\bar{\BF_q})\rightarrow H^0(X/{\BF_q},\Omega)$. This show that $V$ is of multiplicative type.

By Eichler-Shimura, for any $l\nmid pq$, the Frobenius at $l$ acting on $V$ has eigenvalue $1$ or $l$. But as $q$ is odd, $Frob_l$ can not have eigenvalue  $1$ on $\mu_q$, So all the eigenvalues are $l$. Then by (\cite{Mz}, Chapter1, Lemma3.5), we have $V\cong (\mu_q)^{\bigoplus d}$ for some $d$.

By (\cite{Mz}, Chapter2, Lemma7.7), we deduce that $dim_{\BF_q}J[m_q](\bar{\BQ})\geq2$, so $d\geq1$.
\end{proof}

\begin{thm}\label{structure}
If $(q,6)=1$, then there is a non-split exact sequence

\[\xymatrix{
  0 \ar[r]& \BZ/q\BZ \ar[r]&J[m_q]\ar[r]& \mu_q\ar[r]&0
  }\]

which determines the structure of $J[m_q]$ as the unique (up to scale) non-trivial element in $Ext^1_S(\mu_q,\BZ/q\BZ)$.
\end{thm}
\begin{proof}
Suppose the $d$ in \~ref{multiplicative type}is strictly larger than $2$, then we will have a group scheme $G/S$ and an exact sequence

\[\xymatrix{
  0 \ar[r]& \BZ/q\BZ \ar[r]&G\ar[r]& (\mu_q)^{\bigoplus2}\ar[r]&0
  }\]

Let $f_i:\mu_q\rightarrow (\mu_q)^{\bigoplus2}$ ($i=1,2$) be the two embedding. The pull backs will give us two elements in $Ext^1_S(\mu_q,\BZ/q\BZ)$. But by (\cite{BK}, Proposition4.2.1), $dim_{\BF_q}Ext^1_S(\mu_q,\BZ/q\BZ)=1$ (as $(q,6)=1$ and $p=\pm1\pmod q$). So some linear combination of $f_1$ and $f_2$ will gives a sub-group scheme $H$ of $J[m_q]$ which is isomorphic to $\BZ/q\BZ\bigoplus\mu_q$. In particular, there is an embedding of $\mu_q$ in $J[m_q]$.

By \cite{V}, we should have this $\mu_q$ contained in $\sum$-the Shimura subgroup of $J$. But by \cite{LO}, $T_p$ acts on $\sum$ as multiplication by $p$, which contradicts that $\mu_q\subseteq J[m_q]$ which is annihilated by $T_p$.
\end{proof}

Remark:This theorem implies that the  Galois representation given by the action of $G_{\BQ}$ is a two dimensional mod $q$ reducible Galois representation, which is not semisimple.

\subsection{Gross curves and the $2$-Eisenstein quotient}
Recall that for $p$ be a prime which is congruent to $3$ modulo $4$, there is a unique isogeny class of $\BQ$-curves
over $H$ with CM by $\CO$ and the associated character $\phi$.

Let $f_{\phi}(z)=\sum_{(\mathfrak{a},p)=1}\phi(\mathfrak{a})\cdot e^{2\pi i\cdot N_{K/\BQ}(\mathfrak{a})\cdot z}=\sum_{n\geq1}a_n q^n$ ($z\in\mathcal{H},q=e^{2\pi iz}$). By Lemma3 of \cite{Sh1}, $f_{\phi}(z)$ is an eigenform  in $ S_2(\Gamma_0(p^2))$. Let $T$ be the field generated by the image of $\phi$. Then $T$ is a CM field with $T^{+}=\BQ(\{a_n\})$. By theorem7.14 and theorem7.15 of \cite{Sh2}, there is a sub-abelian variety $i:A\rightarrow J$ over $\BQ$ and an embedding $\theta: T^{+}\rightarrow End_{\BQ}(A)$, such that $T_n|_A=\theta(a_n)$ for any $n$, where $T_n$ is the Hecke operator.

\begin{prop}\label{Q-curve}
There is a Gross curve such that $Res_{F/\BQ}E\cong A$, where $F=\BQ(j(E))$.
\end{prop}
\begin{proof}
By Theorem1 of \cite{Sh1}, $A$ is isogenous to $E'^{\bigoplus h}$ for some elliptic curve with CM by $\CO$. Then we have also $A$ isogenous to $(E'^{\sigma})^{\bigoplus h}$ for any $\sigma\in Gal(H/\BQ)$. So $E'$ is isogenous to $E'^{\sigma}$ for any $\sigma\in Gal(H/\BQ)$, i.e. $E'$ is a $\BQ$-curve (note that $\overline{\BQ}$-isogeny is automatically $H$-isogeny).

It is clear that there is a $\BQ$-morphism between $\prod_{\sigma}(E')^{\sigma}$ and $A$. Because $\prod_{\sigma}(E')^{\sigma}$ is simple over $\BQ$, this morphism must be an isogeny. Then, modulo the kernel, we find an $E$ such that $Res_{F/\BQ}E=\prod_{\sigma}(E)^{\sigma}\cong A$.

As $L(E/F,s)=L(s,\chi_E)=L(s,A/\BQ)=\prod_{\tau}L(s,f^{\tau})=L(s,\chi)$ (up to finite Euler factors), we have $\chi_E=\chi$.
\end{proof}

\begin{prop}\label{Gross curve and 2-Eisenstein}
Notations as above. If $p=7\pmod8$, then $A\hookrightarrow\widetilde{J}^{(2)}_{+}$. In particular, there is a non-trivial morphism $E\rightarrow\widetilde{J}^{(2)}_{+}$ for any Gross curve $E$.
\end{prop}
\begin{proof}
This is because $A(\BQ)_{tor}\simeq\BZ/2\BZ$ and the $\epsilon$-factor of its L-function is $1$, when $p=7\pmod8$ (see \cite{G3}).
\end{proof}

\subsection{Heegner point on $\widetilde{J}^{(q)}$ for odd $q$}
\begin{lem}
Let $\eta$ be the Dedkind $\eta$-function and $f=\frac{\eta(pz)^{p+1}}{\eta(z)\eta(p^2z)^p}$. Then $f$ is a rational function on $X$ defined over $\BQ$, and $div(f)=n(P_p-(p-1)[\infty])$.
\end{lem}
\begin{proof}
compute by $\eta$-quotient theory.
\end{proof}

As in \cite{G}, for any field $F$, we can define a homomorphism

$\delta_F:J(F)\rightarrow F^{\times}\bigotimes \BZ/n\BZ$

which sending any $\sum a_i[x_i]\in Div^0(X)(F)$ disjoint from $P_p-(p-1)[\infty]$ to $\prod f(x_i)^{a_i}$.

Recall that the action of $\mathbb{T}$ on $C_p$ gives a homomorphism $\BT/\BI\rightarrow \BZ/n\BZ$. We will view $\BZ/n\BZ$ as a $\BT$-module in this way.

\begin{lem}\label{Hecke algebra homomorphism}
For any field $F$, the map $\delta_F$ is a $\BT$-module homomorphism.
\end{lem}
\begin{proof}
It is enough to check the generators $T_l$ for primes $l$.

Suppose first that $l\neq p$. Let $\alpha_l,\ \beta_l:X_0(lp^2)\rightarrow X$ be the two morphisms sending $(E,C,D)$ ($C$ the $l$-part and $D$ the $p^2$-part) to  $(E,D)$ and $(E/C,(C+D)/D)$ respectively, then $T_l$ is by definition $\alpha_{l,*}\circ \beta^*_l$.

By the construction, $f(T_lz)=f(lz)\prod^{l-1}_{i=0}f(\frac{z+i}{l})=\alpha_{l,*}\circ \beta^*_l(f)$. As $div(\alpha_{l,*}\circ \beta^*_l(f))=\alpha_{l,*}\circ \beta^*_l[div(f)]=T_l(div(f))=(l+1)\cdot div(f)$, we have $f(T_lz)=f^{(l+1)}(z)$ up to constant. When take value on zero divisors, the effect of constants disappears, so we are done.

For $T_p$, recall that $\eta(z)=q^{\frac{1}{24}}\prod^{\infty}_{n=1}(1-q^n)$ ($q=e^{2\pi iz}$), so we have

\[f(T_pz)=\prod^{p-1}_{k=0}f(\frac{z+k}{p})\]

\[=\prod^{p-1}_{k=0}\frac{\eta(z+k)^{p+1}}{\eta(\frac{z+k}{p})\cdot\eta(pz+pk)^p}\]

\[\approx \prod^{p-1}_{k=0}\frac{\eta(z)^{p+1}}{\eta(\frac{z+k}{p})\cdot\eta(pz)^p}\] ($\approx$ means "equal up to constant")

\[=[\frac{\eta(z)}{\eta(pz)}]^{p^2-1}\cdot\frac{\eta^{p+1}}{\prod^{p-1}_{k=0}\eta(\frac{z+k}{p})\cdot\eta(pz)}\]

But we have

\[\eta(z)^{p+1}=q^{\frac{1+p}{24}}\prod^{\infty}_{n=1}(1-q^n)^{1+p}\]

\[\eta(pz)=q^{\frac{p}{24}}\prod^{\infty}_{n=1}(1-q^{pn})\] and

\[\prod^{p-1}_{k=0}\eta(\frac{z+k}{p})\approx q^{\frac{1}{24}}\prod^{p-1}_{k=0}\prod(1-q^{\frac{n}{p}}\zeta^{kn})\] ($\zeta$ a primitive $p$-th root of unity)

\[=q^{\frac{1}{24}}\cdot[\prod_{p\mid n}(1-q^{\frac{n}{p}})^p]\cdot[\prod_{p\nmid n}(1-q^n)]\]

So
\[f(T_pz)\approx[\frac{\eta(z)}{\eta(pz)}]^{p^2-1}\]
which implies that \[\delta_F\circ T_p=0\]

\end{proof}

By the Lemma just proved, we have, for any field, a $\BT$-module homomorphism
\[\delta_F:J(F)\bigotimes\BT/\BI\rightarrow F^{\times}\bigotimes\BZ/n\BZ\]
and hence for any $q|n$
\[\delta_{F,q}:J(F)\bigotimes\BT_{m_q}/\BI\BT_{m_q}\rightarrow F^{\times}\bigotimes\BZ/q^{n_q}\BZ\]
where $\BT_{m_q}$ is the completion of $\BT$ at $m_q$ and $n_q=ord_q(n)$.
\\

Now let $K$ be an imaginary quadratic field such that $(p)=\mathfrak{p}\bar{\mathfrak{p}}$ splits in $K$.

Let $x_K=(\BC/\CO_K\rightarrow\BC/{\fp}^{-1})\in X(H_K)$ and $y_K=\sum_{\sigma\in G(H_K/\BQ)}\epsilon(\sigma)\cdot x^{\sigma}_K\in J(K)^{-}$, where $\epsilon$ is the quadratic character corresponding to $K$. For any $q|n$, let $y^{(q)}_K$ be the projection of $y_K$ in $J^{(q)}(K)^{-}$.

\begin{thm}\label{Heegner point on q-Eisenstein}
Suppose $(q,6)=1$ and $q|(p+1)$. Let $h=h_k/o(\fp)$. Then if $ord_q(h)<ord_q(n)$, then $y^{(q)}_K$ is a non-torsion point in $J^{(q)}(K)^{-}$.
\end{thm}
\begin{proof}
By definition, we have

\[\delta_{K,q}(y_K)=\prod_{\fa\in Cl(\CO_K)}[\frac{\Delta(\fp\fa)^{p+1}}{\Delta(\fa)\Delta(\fp^2\fa)^p}\cdot\frac{\Delta(\overline{\fa})\Delta(\overline{\fp^2\fa})^p}{\Delta(\overline{\fp\fa})^{p+1}}]^{\frac{1}{24}}\]

As $\frac{\Delta(\fp\fa)^{p+1}}{\Delta(\fa)\Delta(\fp^2\fa)^p}=[\frac{\Delta(\fp\fa)}{\Delta(\fp^2\fa)}]^p\frac{\Delta(\fp\fa)}{\Delta(\fa)}$ and for any ideal $\fb\in Cl(\mathcal{O}_K)$, $\frac{\Delta(\fp\fa)}{\Delta(\fa)}$ is a number in $H_K$ generates $\fp^{-12}$, we find that \[\delta_{K,q}(y_K)=u\cdot\alpha^{\frac{p-1}{2}h}\pmod K^{\times n}\]
with $u\in \mathcal{O}^{\times}_K$ and $\alpha\in K^{\times}$ such that $(\alpha)=(\fp/\bar{\fp})^{o(\fp)}$.

Note that $\alpha$ is not a $q$-th power in $K$. This is because if $x^q=\alpha$, then $(x)=(\fp/\bar{\fp})^{o(\fp)/q}$. As $\bar{\fp}$ is the converse of $\fp$ in the ideal class group, we find that $o(\fp)|\frac{2 o(\fp)}{q}$, which is impossible as $q$ is odd.

So when $ord_q(h)<ord_q(n)$, we will have $\delta(y_K)\neq 0$, hence $y_k\neq 0$ in $J(K)^{-}\bigotimes\mathbb{T}_{m_q}$.

But by \~ref{structure}, one easily see that $J(K)^{-}[m_q]=0$, so all points in $J(K)^{-}\bigotimes\mathbb{T}_{m_q}$ is not torsion. This completes the proof.

\end{proof}

\subsection{Heegner point on Gross curves when $p\equiv7\pmod8$}
Let $i:A\hookrightarrow J$ be the sub-abelian variety as in section3.3 and $E$ the Gross curve as in Theorem~\ref{Q-curve}. Let $\pi: J\rightarrow A$ be the dual of $i$, where we have identify the dual of $A$ and $J$ with themselves.

Let $R_i=mC_i$ be of exact order $2$, for $i=1,p$.

\begin{lem}\label{R_1=R_p}
$R_1=R_2$.
\end{lem}
\begin{proof}
By Theorem 1 of \cite{Ling}, the prime to $p$ part of the $\BQ$-rational cuspidal divisor subgroup is isomorphic to $(\BZ/a\BZ)^{2}\bigoplus(\BZ/b\BZ)$, where $a=\frac{p-1}{(p-1,24)}$ and $b=\frac{p+1}{(p+1,24)}$. But as $p=7\pmod8$, $a$ is odd. So the $2$-part of the $\BQ$-rational cuspidal divisor subgroup is cyclic, hence the result.
\end{proof}

By Theorem 22.1.1 of \cite{G3}, we know that $A(\BQ)=<P>\cong \frac{\BZ}{2\BZ}$.

\begin{lem}\label{Hecke on R_i}
$R_1\in J[m_2]$.
\end{lem}
\begin{proof}
Only need to check this for $T_p$. But $T_p(R_1)=T_p(mC_1)m(C_1+C_p)=0$ as $T_p=\sum_{0\leq j\leq p-1}\left(
                                                                                            \begin{array}{cc}
                                                                                              1 & j \\
                                                                                              0 & p \\
                                                                                            \end{array}
                                                                                          \right)
$, so the result follows from Lemma~\ref{R_1=R_p}.
\end{proof}

\begin{prop}\label{Key}
$P\in J[m_2]$
\end{prop}
\begin{proof}
By Shimura's theorems, we know that $T_p=0$ on $A$ as $a_p=0$, and $(l+1)-T_l(P)=(l+1-\theta(a_l))(P)$ for $l\neq p$. To prove the theorem, it is sufficient to show $2\mid deg(l+1-\theta(a_l))$. By the construction of $A$, $deg(l+1-\theta(a_l))=|det(l+1-\theta(a_l))|^2=N_{T^{+}/\BQ}(l+1-\theta(a_l))$. As $N_{T^{+}/\BQ}(l+1-\theta(a_l))=\#A\pmod l(\BF_l)$, the result follows form Lemma~\ref{injective}.
\end{proof}

\begin{lem}\label{injective}
$A(\BQ)[2]\hookrightarrow A\pmod l(\BF_l)$ for any $l\neq p$.
\end{lem}
\begin{proof}

As $A$ has good reduction at $l$ when $l\neq p$, we know that $A(\BQ)[2]\hookrightarrow A\pmod l(\BF_l)$ for any $l\nmid 2p$.

So we only need to show $A(\BQ)[2]\hookrightarrow A\pmod 2(\BF_2)$.

Consider the sub-abelian variety $E$ of $A$ over $H$ as in Prop~\ref{Q-curve}. Let $w$ be a place of $H$ over $2$ and $\fp$ the prime of $K$ below $w$.
Suppose $P\pmod2=0$. Then in terms of $E$, we have $P\in\widehat{E}$ where $\widehat{E}$ is the formal group at $w$. But $\widehat{E}$ is a lubin-Tate formal group, which implies that $P\in E[\fp^{\infty}]$. This contradicts that $P\in E(F)$ and $F=H^{\tau}$ where $\tau$ is the complex multiplication.
\end{proof}

\begin{thm}\label{Heegner point on Gross curve}
Suppose $p=7\pmod8$, $K$ an imaginary quadratic field in which $p$ splits and $y_K$ be the Heegner point. If $\widetilde{J}^{(2)}$ is simple and the class number of $K$ is odd, then the projection of $y_K$ on $A$ is not torsion.
\end{thm}
\begin{proof}
When $p=7\pmod8$, the $\epsilon$-factor is $1$. So to apply Birch's Lemma, we only need to verify $\pi(C_1)$ is not in $2\cdot A(\BQ)$.

Suppose $\pi([0]-[\infty])\in2\cdot A(\BQ)$, then $\pi(C_1)=0$ as we know that $A(\BQ)\simeq\BZ/2\BZ$. Let $B=ker(\pi)$, then we have $P\in B[m_2]$ by Proposition~\ref{Key}. So there is a newform $g\in S^{new}(\Gamma_(p^2))$ such that $A_g\subseteq B$ and $A_g[m_2]\neq0$, in particular this sub-abelian variety $A_g$ will be contained in $\widetilde{J}^{(2)}$.

But $A\nsubseteq B$ as the composition $\pi\circ i$ is multiplication by $deg(\pi)$, so $A\neq A_g$ which contradicts our assumption that $\widetilde{J}^{(2)}$ is simple.
\end{proof}


\begin{thebibliography}{10}
\bibitem{BK} A.Brumer, K.Krammer, Paramodular abelian varieties of odd degree, Transactions of the American Mathematical Society, 2012, 366(5):769-77

\bibitem{BK} A.Brumer and K.Kramer, The rank of elliptic curves, Duke Math.J.(4) 44, (1977), 715-743

\bibitem{G} B.Gross, Heegner point on $X_0(N)$, Modular forms (Durham), 1982

\bibitem{G2} B.Gross, Heights and the specal value of L-series, Number Theory, 1985, 27(10):115-187

\bibitem{G3} B.Gross, Arithmetic of elliptic curves with complex multiplication, Lecture Notes in Mathematics, 1980, 776(1):327-337

\bibitem{K} N.Katz, P-adic properties of modular schemes and modular forms, Lecture Notes in Mathematics350, pp 69-190

\bibitem{Ling} S.Ling, Rational cuspidal subgroup of $J_0(p^r)$, Israel Journal of Mathematics, 1997, 99(1):29-54

\bibitem{LO} S.Ling, J.Osterele, Shimura subgroup of $J_0(N)$, Ast¨¦risque196¨C197 (1991), 171¨C203.

\bibitem{Li} G.Ligozat, Courbes modulaires de genre 1, Bull. Soc. Math. Frace Mem. 43(1975)

\bibitem{Mum} D.Mumford, J.Forgarty, Kirwan, Geometric invarint theory, Springer, 1965, 93(4):99-127

\bibitem{Mumford} D.Mumford, Abelian variety, Tata Institute of Fundamental Research Studies in Matheamtics, 1970, 27(7-8):338-354

\bibitem{Mz} B.Mazur, Eisenstein ideals and modular curves, Publications math¨¦matiques de l'IH¨¦S, 1977, 47(1):33-186

\bibitem{Mz2} B.Mazur, On the Arithmetic of Special Value of L Function, Inventiones mathematicae, 1979, 55(3):207-240

\bibitem{NS} NS-curve

\bibitem{R} K.Ribet, Mod p Hecke operators and congruences between modular forms, Inventiones mathematicae, 1983, 71(1):193-205

\bibitem{Sh1} G.Shimura, Elliptic curves with CM as factors of Jacobians of modular fuctions, Nagoya Mathematical Journal, 1971, 43:199-208

\bibitem{Sh2} G.Shimura, Introduction to the artithmetic theory of automorphic functions, Publications of the Mathematical Society of Japan, 11

\bibitem{V} Vatsal, Multiplicative subgroup of $J_0(N)$, Journal of the Institute of Mathematics of Jussieu, 2005, 4(2):281-316

\bibitem{Y} T.Yang, Nonvanishing of certein Hecke L-series and rank of certein elliptic curves, compositio math.117 (1999), No.3, 337-359.


\end{thebibliography}
\end{document}